\documentclass[aap,MSNbibl,nameyear,seceqn,dvips]{arximspdf}
\usepackage{algorithm}
\usepackage{algpseudocode}
\usepackage{graphicx}

\doi{10.1214/13-AAP982} 
\volume{24}
\issue{6}
\pubyear{2014}
\firstpage{2455}
\lastpage{2490}

\makeatletter

\newcommand{\rrVert}{\Vert}
\newcommand{\rrvert}{\vert}
\newcommand{\llVert}{\Vert}
\newcommand{\llvert}{\vert}

\newproclaim{defn}{Definition}[section]
\newproclaim{rem}{Remark}
\newtheorem{prop}[defn]{Proposition}
\newtheorem{lem}[defn]{Lemma}
\newproclaim{assumption}[defn]{Assumption}
\newtheorem{theorem}[defn]{Theorem}
\def\TV{\mathrm{TV}}

\newcommand{\diam}{\operatorname{diam}}
\newcommand{\eqref}[1]{(\ref{#1})}
\renewcommand{\epsilon}{\varepsilon}
\def\sfrac#1#2{#1/#2}

\def\afrac#1#2{#1/(#2)}

\makeatother

\begin{document}
\begin{frontmatter}

\title{Spectral gaps for a Metropolis--Hastings algorithm in infinite dimensions}
\runtitle{Spectral gaps for MCMC in infinite dimensions}

\begin{aug}
\author[A]{\fnms{Martin} \snm{Hairer}\thanksref{MH}\ead[label=e1]{m.hairer@warwick.ac.uk}},
\author[A]{\fnms{Andrew M.} \snm{Stuart}\thanksref{AMS}\ead[label=e2]{a.m.stuart@warwick.ac.uk}}
\and\break
\author[B]{\fnms{Sebastian J.} \snm{Vollmer}\corref{}\thanksref{SJV}\ead[label=e3]{s.vollmer@warwick.ac.uk}}
\runauthor{M. Hairer, A. M. Stuart and S. J. Vollmer}
\affiliation{University of Warwick}
\address[A]{M. Hairer\\
A. M. Stuart\\
Mathematical Institute\\
University of Warwick\\
Coventry, CV4 7AL\\
United Kingdom\\
\printead{e1}\\
\hphantom{E-mail: }\printead*{e2}} 
\address[B]{S. J. Vollmer\\
Mathematical Institute\\
University of Warwick\\
Coventry, CV4 7AL\\
United Kingdom\\
Present address:\\
Department of Statistics\\
1 South Parks Road\\
Oxford, OX1 3TG\\
United Kingdom\\
\printead{e3}}
\end{aug}
\thankstext{MH}{Supported by EPSRC, the Royal Society, and the
Leverhulme Trust.}
\thankstext{AMS}{Supported by EPSRC and ERC.}
\thankstext{SJV}{Supported by ERC.}

\received{\smonth{12} \syear{2011}}
\revised{\smonth{2} \syear{2013}}

%
\begin{abstract}
We study the problem of sampling high and infinite dimensional target
measures arising in applications such as conditioned diffusions and
inverse problems. We focus on those that arise from approximating
measures on Hilbert spaces defined via a density with respect to a
Gaussian reference measure. We consider the Metropolis--Hastings algorithm
that adds an accept--reject mechanism to a Markov chain proposal in
order to make the chain reversible with respect to the target measure.
We focus on cases where the proposal is either a Gaussian random walk
(RWM) with covariance equal to that of the reference measure or an
Ornstein--Uhlenbeck proposal (pCN) for which the reference measure
is invariant.

Previous results in terms of scaling and diffusion limits suggested
that the pCN has a convergence rate that is independent of the dimension
while the RWM method has undesirable dimension-dependent behaviour.
We confirm this claim by exhibiting a dimension-indepen\-dent Wasserstein
spectral gap for pCN algorithm for a large class of target measures.
In our setting this Wasserstein spectral gap implies an $L^{2}$-spectral
gap. We use both spectral gaps to show that the ergodic average satisfies
a strong law of large numbers, the central limit theorem and nonasymptotic
bounds on the mean square error, all dimension independent. In contrast
we show that the spectral gap of the
RWM algorithm applied to the reference measures degenerates
as the dimension tends to infinity.
\end{abstract}

%

\end{frontmatter}
\newpage

%
\section{Introduction}\label{introduction}
The aim of this article is to study the complexity of certain sampling
algorithms in high dimensions. Creating samples from a high dimensional
probability distribution\vadjust{\goodbreak} is an essential tool in Bayesian inverse problems
[\citet{MR2652785}], Bayesian statistics [\citet{lee2004bayesian}], Bayesian
nonparametrics [\citet{BayesianNonBook}], and conditioned diffusions
[\citet{nonlinearsampling}]. For example, in inverse problems, some
input data such as initial conditions or parameters for a forward
mathematical model have to be determined
from observations of noisy output. In the Bayesian approach, assuming
a prior on the unknown input, and conditioning on the data, results
in the posterior distribution, a natural target for
sampling algorithms. In fact these sampling algorithms are also used in
optimisation
in form of simulated annealing [\citet
{geyer1995annealing,2011arXiv1108.1494P}].

The most widely used method for general target measures are Markov
chain Monte Carlo (MCMC) algorithms which use a Markov chain that
in stationarity yields dependent samples from the target. Moreover,
under weak conditions, a law of large numbers holds for the empirical
average of a function $f$ (observable) applied to the steps of the
Markov chain. We quantify the computational cost of such an algorithm
as
\[
\mbox{number of necessary steps }\times\mbox{ cost of a step.}
\]

While for most algorithms the cost of one step grows with the dimension,
a major result of this article is to exhibit an algorithm which, when
applied to measures defined via a finite-dimensional approximation of
a measure defined by a density with respect to a Gaussian random field,
requires a number of steps independent of the dimension in order to
achieve a given level of accuracy.

For ease of presentation we work on a separable Hilbert space
$(\mathcal{H}, \langle\cdot,\cdot \rangle)$
equipped with a mean-zero Gaussian reference measure $\gamma$ with
covariance operator $\mathcal{C}$. Let $\{\varphi_{n}\}{}_{n\in
\mathbb{N}}$
be an orthonormal basis of eigenvectors of $\mathcal{C}$ corresponding
to the eigenvalues $\{\lambda_{n}^{2}\}_{n\in\mathbb{N}}$. Thus
$\gamma$
can be written as its Karhunen--Lo\`{e}ve expansion [\citet{karhunen}]
\[
\gamma=\mathcal{L}\Biggl(\sum_{i=1}^{\infty}
\lambda_{i}e_{i}\xi_{i}\Biggr)\qquad  \mbox{where }
\xi_{i}\stackrel{\mathrm{i.i.d.}} {\sim }\mathcal{N}(0,1)
\]
and where $\mathcal{L}(\cdot)$ denotes the law of a random variable.
The target measure $\mu$ is assumed to have a density with respect
to $\gamma$ of the form
%
\begin{equation}
\mu(dx)=M\exp\bigl(-\Phi(x)\bigr)\gamma(dx).\label{eq:target}
\end{equation}
With $P_{m}$ being the projection onto the first $m$ basis elements,
we consider the following $m$-dimensional approximations to $\gamma$
and $\mu$:
%
\begin{eqnarray}\label
{eq:finitetarget}
\gamma_{m}(dx) & = & \mathcal{L}\Biggl(\sum
_{i=1}^{m}\lambda_{i}e_{i}\xi
_{i}\Biggr) (dx),
\nonumber
\\[-8pt]\\[-8pt]
\mu_{m}(dx) & = & M_{m}\exp\bigl(-\Phi(P_{m}x)
\bigr)\gamma_{m}(dx).\nonumber
\end{eqnarray}
The approximation error, namely the difference between $\mu$ and
$\mu_{m}$, is already well studied [\citet
{UncertaintyElliptic,MCMCinHighDimensions}] and can be estimated in
terms of the closeness between
$\Phi\circ P_{m}$ and $\Phi$.

In this article we consider Metropolis--Hastings MCMC methods [\citet
{metropolis1953equation,hastings1970monte}].
For an overview of other MCMC methods, which have been developed and
analysed, we refer the reader to \citet
{Robert2004MonteCarlo,liu2008monte}. The idea of
the Metropolis--Hastings algorithm is to add an accept--reject mechanism
to a Markov chain proposal in order to make the resulting Markov
chain reversible with respect to the target measure.
We denote the transition
kernel of the underlying Markov chain by $Q(x,dy)$ and the
acceptance probability for a proposed move from $x$ to $y$ by $\alpha
(x,y)$. The
transition kernel of the Metropolis--Hastings algorithm reads
%
\begin{equation}\label{eq:MetropolisKernel}
\mathcal{P}(x,dz)=Q(x,dz)\alpha(x,z)+\delta_{x}(dz)\int\bigl(1-\alpha
(x,u)\bigr)Q(x,du),
\end{equation}
where $\alpha(x,y)$ is chosen such that $\mathcal{P}(x,dy)$ is reversible
with respect to $\mu$ [\citet{samplingFirstInfiniteDimensional}]. For
the random walk Metropolis algorithm (RWM) the proposal kernel corresponds
to
\[
Q(x,dy)=\mathcal{L}(x+\sqrt{2\delta}\xi) (dy)
\]
with $\xi\sim\gamma_{m}$ which leads to the following acceptance probability:
%
\begin{equation}\label{eq:defRWM}
\alpha(x,y)=1\wedge \bigl(\Phi(x)-\Phi(y)+\tfrac{1}{2}\bigl\langle x,
\mathcal{C}^{-1}x\bigr\rangle-\tfrac{1}{2}\bigl\langle y,\mathcal
{C}^{-1}y\bigr\rangle \bigr).
\end{equation}

Notice that the quadratic form $\frac{1}{2}\langle y,\mathcal
{C}^{-1}y\rangle$
is almost surely infinite with respect to the proposal because it
corresponds to the Cameron--Martin norm of $y$. For this reason the
RWM algorithm is not defined on the infinite dimensional Hilbert space
$\mathcal{H}$ [consult \citet{CotterNumericalResultPRWM} for a discussion],
and we will study it only on $m$-dimensional approximating spaces.
In this article we will demonstrate that the RWM can be considerably
improved by using the preconditioned Crank--Nicolson (pCN) algorithm
which is defined via
%
\begin{eqnarray}\label{eq:defPRWM}
Q(x,dy) & = & \mathcal{L}\bigl((1-2\delta)^{\sfrac{1}{2}}x+\sqrt
{2\delta}\xi\bigr),
\\
\alpha(x,y) & = & 1\wedge\exp\bigl(\Phi(x)-\Phi(y)\bigr)
\end{eqnarray}
with $\xi\sim\gamma$. The pCN was introduced in \citet{beskos2008mcmc}
as the PIA algorithm, in the case $\alpha=0$. Numerical
experiments in \citet{CotterNumericalResultPRWM} demonstrate its favourable
properties in comparison with the RWM algorithm. In contrast to RWM,
the acceptance probability is well defined on a Hilbert space, and this
fact gives an intuitive explanation for the theoretical results derived
in this paper in which we develop a theory explaining the
superiority of pCN over RWM when applied on sequences of
approximating spaces of increasing dimension.
Our main positive results about pCN can
be summarised in the following way (rigorous statements in Theorems
\ref{thm:localVaryingBalls},
\ref{thm:wassersteinImplL2}, \ref{thm:slln} and \ref{thm:wassersteinclt}):
\begin{claim*}
Suppose that both $\Phi$ and its local Lipschitz constant satisfy
a growth assumption at infinity. Then for a fixed $0<\delta\leq\frac{1}{2}$,
the pCN algorithm applied to $\mu_m(\mu)$:
\begin{longlist}[III.]
\item[I.] has a unique invariant measure $\mu_m$ ($\mu$);
\item[II.] has a Wasserstein spectral gap uniformly in $m$;
\item[III.] has an $L^{2}$-spectral gap $1-\beta$ uniform in $m$.
\end{longlist}
The corresponding sample average $S_{n}(f)=\frac{1}{n}
\sum_{i=1}^{n}f(X_{i})$:
\begin{longlist}[IV.]
\item[IV.] satisfies a strong law of large numbers and a central limit
theorem (CLT) for a class of locally Lipschitz functionals for every
initial condition;
\item[V.] satisfies a CLT for $\mu$ ($\mu_{m}$)-almost every initial
condition
with asymptotic variance uniformly bounded in $m$ for $f\in L_{\mu
}^{2}$  $(L_{\mu_{m}}^{2} )$;
\item[VI.] has an explicit bound on the mean square error (MSE) between
itself and $\mu(f)$ for certain initial distributions $\nu$.
\end{longlist}
\end{claim*}
These positive results about pCN clearly apply to $\Phi=0$ which
corresponds to the target measures $\gamma$ and $\gamma_{m}$, respectively;
in this case the acceptance probability of pCN is always one, and the
theorems mentioned are simply statements about a discretely sampled
Ornstein--Uhlenbeck (OU) process on $\mathcal{H}$ in this case. On
the other hand the RWM algorithm applied to a specific Gaussian target
measure $\gamma_{m}$ has an $L_{\mu}^{2}$-spectral gap which converges
to $0$ as $m\rightarrow\infty$ as fast as any negative power of
$m$; see Theorem~\ref{thm:negativeRWM}.

While it is a major contribution of this article to establish the results
I, II and IV for pCN and to establish the negative results for RWM, the
statements III, V and VI follow by verification of the conditions of
known results.

In addition to the significance of these results in their own right for the
understanding of MCMC methods, we would also like to highlight the
techniques that we use in the proofs. We apply recently developed tools
for the study of Markov chains on infinite dimensional spaces; see \citet
{weakHarris}.
A weak version of Harris's theorem [proved in \citet{weakHarris}] makes\vspace*{1pt} a Wasserstein spectral gap verifiable
in practice, and for reversible Markov processes it even implies an
$L^{2}$-spectral gap. Henceforth, we shall refer to this as the weak Harris theorem.

\subsection{Literature review}\label{literature review}
The results in the literature can broadly be classified as follows
[\citet{explicitbdd, Meyn:2009uqa}]:
\begin{longlist}[(3)]
\item[(1)] For a metric on the space of measures such as the total
variation or the Wasserstein metric, the rate of convergence to
equilibrium can be characterised through the decay of $d(\nu\mathcal
{P}^{n},\mu)$ where $\nu$ is the initial distribution of the Markov chain.
\item[(2)] For the Markov operator $\mathcal{P}$ the convergence rate
is given as the operator norm of $\mathcal{P}$ on a space of functions
from $X$ to $\mathbb{R}$ modulo constants. The most prominent example
here is the $L^{2}$-spectral gap.
\item[(3)] Direct methods like regeneration and the so-called
split-chain which use the dynamics of the algorithm to introduce
independence. The independence can be used to prove central limit theorem.
Previous results have been formulated in terms of the following three
main types of convergence:
\end{longlist}
%
Between these notions of convergence, there are many fruitful relations;
for details consult \citet{explicitbdd}. All these convergence types
have been
used to study MCMC algorithms.

The first systematic approach to prove $L^{2}$-spectral gaps for
Markov chains was developed in \citet{lawler1988bounds} using the
conductance concept due to \citet{cheeger1970lower}. These
results were extended and applied to the Metropolis--Hastings algorithm
with uniform proposal and a log-concave target distribution on a bounded
convex subset of $\mathbb{R}^{n}$ in \citet
{lovasz1993convexBodyLazy}. The consequences
of a spectral gap for the ergodic average in terms of a CLT and the
MSE have been investigated in \citet{kipnis1986central,cuny2009pointwise}
and \citet{explicitbdd}, respectively, and were first brought up in
the MCMC literature in \citet{geyer1992practical,chan1994discussion}.

For finite state Markov chains the spectral gap can be bounded in
terms of quantities associated with its graph [\citet
{diaconis1991geometric}]. This idea has also been applied to the
Metropolis-algorithm in
\citet{sinclair1989approximate} and \citet{frigessi1993convergence}.

A different approach using the splitting chain technique
mentioned above
was independently developed in \citet{nummelin1978splitting} and
\citet{MR511425}
to bound the total variation distance between the $n$-step kernel
and the invariant measure. Small and petite sets are used in order
to split the trajectory of a Markov chain into independent blocks.
This theory was fully developed in \citet{Meyn:2009uqa} and again
adapted and applied to the Metropolis--Hastings algorithm in \citet
{roberts1996geometric}
resulting in a criterion for geometric ergodicity
\[
\bigl\llVert \mathcal{P}(x,\cdot)^{n}-\mu\bigr\rrVert _{\TV}
\leq C(x)c^{n} \qquad \mbox{for some }c<1.
\]
Moreover, they established a criterion for a CLT. Extending this
method, it was also possible to derive rigorous confidence intervals
in \citet{2011KrysConfidence}.

In most infinite dimensional settings, the splitting chain method cannot
be applied since measures tend to be mutually singular. The method
is hence not well-adapted to the high-dimensional setting. Even Gaussian
measures with the same covariance operator are only equivalent if
the difference between their means lies in the Cameron--Martin space.
As a consequence, the pCN algorithm is not irreducible in the sense
of \citet{Meyn:2009uqa}, hence there is no nontrivial measure
$\varphi$
such that $\varphi(A)>0$ implies $\mathcal{P}(x,A)>0$ for all $x$.
By inspecting the Metropolis--Hastings transition kernel (\ref
{eq:MetropolisKernel}),
the pCN algorithm is not irreducible. More precisely if $x-y$ is not in
the Cameron--Martin space $Q(x,dz)$ and $Q(y,dz)$ are mutually
singular, consequently the same is true for $P(x,dz)$ and $P(y,dz)$.
This may also be shown to be true for the $n$-step kernel by expressing
it as a sum of densities times Gaussian measures and applying the
Feldman--Hajek Theorem [\citet{StochEqnInf}].

For these reasons, existing theoretical results concerning RWM and
pCN in high dimensions have been confined to scaling results and derivations
of diffusion limits. In \citet{mcmcLocalScaling} the RWM algorithm
with a target
that is absolutely continuous with respect to a product measure has
been analysed for its dependence on the dimension. The proposal distribution
is a centred normal random variable with covariance matrix $\sigma_{n}I_{n}$.
The main result there is that $\delta$ has to be chosen as a constant
times a particular negative power of $n$ to prevent the expected
acceptance probability to go to one or to zero. In a similar setup it
was recently shown that there is a $\mu$-reversible
SPDE limit if the product law is a truncated Karhunen--Lo\`{e}ve expansion
[\citet{MCMCinHighDimensions}].
This SPDE limit suggests that the number of steps necessary for a
certain level of accuracy grows like $\mathcal{O}(m)$ because
$\mathcal{O}(m)$ steps are necessary in
order to approximate the SPDE limit on $[0,T] $. A similar result in
\citet{2011arXiv1108.1494P}
suggests that the pCN algorithm only needs $\mathcal{O}(1)$ steps.

Uniform contraction in a Wasserstein distance was first applied
to MCMC in \citet{Ollivier2010Curvature} in order to get bounds on the variance
and bias of the sample average of Lipschitz functionals. We use the
weak Harris theorem to verify this contraction, and by using the results
from \citet{explicitbdd}, we obtain nonasymptotic bounds on the sample
average of $L_{\mu}^{2}$-functionals.
In \citet{2012arXiv1210.1180E} \mbox{exponential} convergence for a Wasserstein
distance is proved for the Metropolis-adjusted-Langevin (MALA) and
pCN algorithm for log-concave measures having a density with respect
to a Gaussian measure. The rates obtained in this article are explicit
in terms of additional bounds on the derivates of the density. In
our proofs we do not assume log-concavity. However, the rate obtained
here is less explicit.

Similarly, approaches based on the Bakry--Emery criterion [\citet{MR889476}]
seem to be only applicable if the measure is log-concave.

\subsection{Outline}
In this paper we substantiate these ideas by using spectral\linebreak[4]  gaps
derived by an application of the weak Harris theorem [\citet{weakHarris}].
Section~\ref{main results} contains the statements of our main
results, namely Theorems
\ref{thm:global}, \ref{thm:localVaryingBalls} and \ref
{thm:wassersteinImplL2} concerning the desirable dimension-independence
properties of the pCN method and Theorem~\ref{thm:negativeRWM} dealing
with the undesirable
dimension dependence of the RWM method. Section~\ref{main results}
starts by specifying
the RWM and pCN algorithms as Markov chains, the statement of the weak
Harris theorem, and a discussion of the relationship between exponential
convergence in a Wasserstein distance and $L_{\mu}^{2}$-spectral
gaps. The proofs of the theorems in Section~\ref{main results} are
given in Section~\ref{proof for spectral gaps}.
We highlight that the key steps can be found in the Sections~\ref{sub:basicCoupling}
and \ref{sub:Contraction-in-the} where we dealt with the cases of
global and local Lipschitz~$\Phi$,
respectively. In Section~\ref{sample path average} we exploit the
Wasserstein and $L_{\mu}^{2}$-spectral gaps in order to derive a law
of large numbers (LLN), central limit theorems (CLTs), and mean square
error (MSE) bounds for sample-path
ergodic averages of the pCN method, again emphasising the dimension
independence of these results. We draw overall conclusions in Section~\ref{sec:Conclusion}.

\section{Main results}\label{main results}

In Section~\ref{sub:Algorithms} we specify the RWM and pCN algorithms
before summarising the weak Harris theorem in Section~\ref{sub:weakharris}. Subsequently, we describe how a Wasserstein spectral
gap implies an $L_{\mu}^{2}$-spectral gap. Based on the weak Harris
theorem, we give necessary conditions on
the target measure for the pCN algorithm in order to have a dimension
independent
spectral gap in a Wasserstein distance in Section~\ref{sub:Results}.
In Section~\ref{sub:RWM}
we highlight one of the disadvantages of the RWM by giving an example
satisfying our assumptions for the pCN algorithm for which the
spectral gap of the RWM algorithm converges to zero as fast as any
negative power of $m$ as $m\to\infty$.

\subsection{Algorithms}
\label{sub:Algorithms}

We focus on convergence results for the pCN algorithm (Algorithm \ref{alg:pCN})
which generates a Markov chain $\{X^{n}\}_{n\in\mathbb{N}}$ with
$X^{n}\in H$
and $\{X_{m}^{n}\}_{n\in\mathbb{N}}$ when it is applied to the
measures $\mu$
and $\mu_{m}$, respectively. The corresponding transition Markov kernels
are called $\mathcal{P}$ and $\mathcal{P}_{m}$, respectively. We
use the same notation for the Markov chain generated by the RWM (Algorithm
\ref{alg:RWM}). This should not cause confusion as the statements concerning the pCN
and RWM algorithms occur in separate sections.
\begin{algorithm}[t]
Initialise $X_{0}$.

For $n\geq0$ do:
\begin{enumerate}
\item Generate $\xi\sim\gamma$ and set $p_{X_{n}}(\xi
)=(1-2\delta)^{\sfrac{1}{2}}X_{n}+\sqrt{2\delta} \xi$.
\item Set
\[
X_{n+1}= %
\cases{ p_{X_{n}}, &\quad  \mbox{with probability}
$\alpha(X_{n},p_{X_{n}}),$
\cr
X_{n}, & \quad \mbox{otherwise}.
} %
\]
Here, $\alpha(x,y)=1\wedge\exp(\Phi(x)-\Phi(y))$.
\end{enumerate}
\caption{Preconditioned Crank--Nicolson}\label
{alg:pCN}
\end{algorithm}
\begin{algorithm}
Initialise $X_{0}$.

For $n\geq0$ do:
\begin{enumerate}
\item Generate $\xi\sim\gamma_{m}$ and set $p_{X_{n}}(\xi
)=X_{n}+\sqrt{2\delta} \xi$.
\item Set
\[
X_{n+1}= %
\cases{ p_{X_{n}}, &\quad  \mbox{with probability}
$\alpha(X_{n},p_{X_{n}})$,
\cr
X_{n}, & \quad \mbox{otherwise}.
} %
\]
Here, $\alpha(x,y)= 1\wedge\exp(\Phi(x)-\Phi(y)+\frac
{1}{2}\langle x,\mathcal{C}^{-1}x\rangle-\frac{1}{2}\langle
y,\mathcal{C}^{-1}y\rangle)$.
\end{enumerate}
\caption{Random walk Metropolis}\label{alg:RWM}
\end{algorithm}

\subsection{Preliminaries}\label{sub:weakharris}

In this section we review Lyapunov functions, Wasserstein distances,
$d$-small sets and $d$-contracting Markov kernels in order to state
the weak Harris theorem of \citet{weakHarris}.
By weakening the notion of small sets, this theorem gives a sufficient
condition for exponential convergence in a Wasserstein distance. Moreover,
we explain how this implies an $L^{2}$-spectral gap.

\subsubsection{Weak Harris theorem}
%
\begin{defn}
Given a Polish space $\mathbf{E}$, a function $d\dvtx \mathbf{E}\times
\mathbf{E}\rightarrow\mathbb{R}_{+}$
is a \textit{distance-like} function if it is symmetric, lower semi-continuous
and $d(x,y)=0$ is equivalent to $x=y$.
\end{defn}

This induces the $1$-Wasserstein ``distance'' associated with $d$
for the measures $\nu_{1}, \nu_{2}$
%
\begin{equation}\label
{eq:liftedMetric}
d(\nu_{1},\nu_{2})  =  \inf_{\pi\in\Gamma(\nu_{1},\nu
_{2})}\int
_{\mathbf{E}\times\mathbf{E}}d(x,y)\pi(dx,dy),
\end{equation}
where $\Gamma(\nu_{1},\nu_{2})$ is the set of couplings of $\nu_{1}$
and $\nu_{2}$ (all measures on $\mathbf{E}\times\mathbf{E}$ with
marginals $\nu_{1}$ and $\nu_{2}$). If $d$ is a metric, the Monge--Kantorovich
duality states that
\[
d(\nu_{1},\nu_{2})=\sup_{\llVert  f\rrVert
_{\operatorname{Lip}(d)}=1}\int f\,d
\nu_{1}-\int f\,d\nu_{2}.
\]

We use the same notation for the distance and the associated Wasserstein
distance; we hope that this does not lead to any confusion.
%
\begin{defn}
\label{def:d-contracting}A Markov kernel $\mathcal{P}$ is \textit
{$d$-contracting}
if there is $0<c<1$ such that $d(x,y)<1$ implies
\[
d\bigl(\mathcal{P}(x,\cdot),\mathcal{P}(y,\cdot)\bigr)\leq c\cdot d(x,y).
\]
\end{defn}
%
\begin{defn}
Let $\mathcal{P}$ be a Markov operator on a Polish space $\mathbf{E}$
endowed with a distance-like function $d\dvtx \mathbf{E}\times\mathbf
{E}\rightarrow[0,1]$.
A set $S\subset\mathbf{E}$ is said to be $d$\textit{-small} if
there exists $0<s<1$ such that for every $x,y\in S$
\[
d\bigl(\mathcal{P}(x,\cdot),\mathcal{P}(y,\cdot)\bigr)\leq s.
\]
\end{defn}
\begin{rem*}
The $d$-Wasserstein distance associated with
\[
d(x,y)=\chi_{\{x\neq y\}}(x,y)
\]
coincides with the total variation distance (up to a factor $2$). If
$S$ is a small set
\citet{Meyn:2009uqa}, then there exists a probability measure
$\nu$ such that $\mathcal{P}$ can be decomposed into
\[
\mathcal{P}(x,dz)=s\tilde{\mathcal{P}}(x,dz)+(1-s)\nu(dz) \qquad \mbox {for }x\in S.
\]
This implies that $d_\TV(\mathcal{P}(x,\cdot),\mathcal{P}(y,\cdot
))\leq s$
and hence $S$ is $d$-small, too.
\end{rem*}
%
\begin{defn}
A Markov kernel $\mathcal{P}$ has a Wasserstein spectral gap if there
is a $\lambda>0$ and a $C>0$ such that
\[
d\bigl(\nu_{1}\mathcal{P}^{n},\nu_{2}
\mathcal{P}^{n}\bigr)\leq C\exp (-\lambda n)d(\nu_{1},
\nu_{2})\qquad  \mbox{for all }n\in\mathbb{N}.
\]
\end{defn}
%
\begin{defn}
$V$ is a \textit{Lyapunov} function for the Markov operator $\mathcal{P}$
if there exist $K>0$ and $0\leq l<1$ such that
%
\begin{equation}\label{eq:defLyapunov}
\mathcal{P}^{n}V(x)\leq l^{n}V(x)+K\qquad \mbox{for all }x\in
\mathbf {E} \mbox{ and all }n\in\mathbb{N}.
\end{equation}
(Note that the bound for $n=1$ implies all other bounds but with a
different constant~$K$.)
\end{defn}
%
\begin{prop}[(Weak Harris theorem {[}\citet{weakHarris}{]})]
\label{thm:weakHarris}
Let $\mathcal{P}$ be a Markov kernel over a Polish space $\mathbf
{E}$, and assume that:
\begin{longlist}[(3)]
\item[(1)] $\mathcal{P}$ has a Lyapunov function $V$ such that (\ref
{eq:defLyapunov})
holds;
\item[(2)] $\mathcal{P}$ is $d$-contracting for a distance-like
function $d\dvtx \mathbf{E}\times\mathbf{E}\rightarrow[0,1]$;
\item[(3)] the set $S=\{x\in\mathbf{E}\dvtx V(x)\leq4K\}$ is $d$-small.
\end{longlist}
Then there exists $\tilde{n}$ such that for any two probability measures
$\nu_{1}$, $\nu_{2}$ on $\mathbf{E}$, we have
\[
\tilde{d}\bigl(\nu_{1}\mathcal{P}^{\tilde{n}},\nu_{2}
\mathcal {P}^{\tilde{n}}\bigr)\leq\tfrac{1}{2}\tilde{d}(
\nu_{1},\nu_{2}),
\]
where $\tilde{d}(x,y)=\sqrt{d(x,y)(1+V(x)+V(y))}$, and $\tilde{n}(l,K,c,s)$
is increasing in $l$, $K$, $c$ and $s$. In particular there is
at most one invariant measure. Moreover, if there exists a complete
metric $d_{0}$ on $\mathbf{E}$ such that $d_{0}\leq\sqrt{d}$ and
such that $\mathcal{P}$ is Feller on $\mathbf{E}$, then there exists
a unique invariant measure $\mu$ for $\mathcal{P}$.
\end{prop}
\begin{rem*}
Setting $\nu_{2}=\mu$ we obtain the convergence rate to the invariant
measure.
\end{rem*}

\subsubsection{The Wasserstein spectral gap implies an
$L^{2}$-spectral gap}
\label{sub:L2}

In this section we give reasons why a Wasserstein spectral gap implies
an $L_{\mu}^{2}$-spectral gap under mild assumptions for a Markov
kernel $\mathcal{P}$. The proof is based on a comparison of different
powers of $\mathcal{P}$ using the spectral theorem.
%
\begin{defn}[($L_{\mu}^{2}$-spectral gap)]
\label{def:L2gap} A Markov operator
$\mathcal{P}$ with invariant measure $\mu$ has an $L_{\mu}^{2}$-spectral
gap $1-\beta$ if
\[
\beta=\llVert \mathcal{P}\rrVert _{L_{0}^{2}\rightarrow
L_{0}^{2}}=\sup_{f\in L_{\mu}^{2}}
\frac{\llVert \mathcal
{P}f-\mu(f)\rrVert _{2}}{\llVert  f-\mu(f)\rrVert _{2}}<1.
\]
\end{defn}

The following proposition is a discrete-time
version of Theorem~2.1(2) in \citet{MR1953493}. The proof given below
is from private communication with Wang and is presented because
of its beauty and the tremendous consequences in combination with
the weak Harris theorem.
%
\begin{prop}[(Private communication {[}\citet{rockner2001weak}{]})]
\label{prop:wassersteinL2}
Let $\mathcal{P}$ be a Markov transition operator which is reversible
with respect to $\mu$ and suppose that $\operatorname{Lip}(\tilde{d})\cap L_{\mu
}^{\infty}$
is dense in $L_{\mu}^{2}$. If for every such $f$ there exists a
constant $C(f)$ such that
\[
\tilde{d}\bigl(\bigl(\mathcal{P}^nf\bigr)\mu,\mu\bigr)\leq C(f)
\exp(-\lambda n)\tilde {d}(f\mu,\mu),
\]
then this implies the $L_{\mu}^{2}$-spectral gap
%
\begin{equation}\label{eq:L2Variation}
\bigl\llVert \mathcal{P}^{n}f-\mu(f)\bigr\rrVert _{2}^{2}
\leq\bigl\llVert f-\mu(f)\bigr\rrVert _{2}^{2}\exp(-\lambda
n).
\end{equation}
\end{prop}
\begin{pf}
First assume that $0\leq f\in \operatorname{Lip}(\tilde d)\cap L^{\infty}(\mu)$ with
$\mu(f)=1$
and $\pi$ being the optimal coupling between $(\mathcal{P}^{2n}f)\mu$
and $\mu$ for the Wasserstein distance associated with $d$. Reversibility
implies $\int(\mathcal{P}^{n}f)^{2}\,d\mu=\int(\mathcal
{P}^{2n}f)f\,d\mu$
which leads to
\begin{eqnarray*}
\bigl\llVert \mathcal{P}^{n}f-\mu(f)\bigr\rrVert _{2}^{2}
& = & \mu \bigl(\bigl(\mathcal{P}^{n}f\bigr)^{2} \bigr)-1=
\int\bigl(f(x)-f(y)\bigr)\,d\pi
\\
& \leq& \operatorname{Lip}(f)\int\tilde{d}(x,y)\,d\pi\leq \operatorname{Lip}(f)\tilde{d}\bigl(\mathcal
{P}^{2n}f\mu,\mu\bigr)
\\
& = & \operatorname{Lip}(f)\tilde{d}\bigl((f\mu)\mathcal{P}^{2n},\mu\bigr)\leq
C\operatorname{Lip}(f)\exp (-2\lambda n).
\end{eqnarray*}

Since the above extends to $a\cdot f$, we note that for general $f\in
L^{\infty}\cap \operatorname{Lip}(\tilde{d})$,
\begin{eqnarray*}
\bigl\llVert P_{t}f-\mu(f)\bigr\rrVert _{2}^{2} &
\leq& 2\bigl\llVert P_{t}f^{+}-\mu\bigl(f^{+}
\bigr)\bigr\rrVert _{2}^{2}+2\bigl\llVert P_{t}f^{-}-
\mu \bigl(f^{-}\bigr)\bigr\rrVert _{2}^{2}.
\end{eqnarray*}

By Lemma~\ref{lem:L2constant}, bound (\ref{eq:L2Variation})
holds for functions in $\operatorname{Lip}\,\cap\, L^{\infty}(\mu)$. Hence the result
follows by taking limits of such functions.
\end{pf}
%
\begin{lem}
\label{lem:L2constant}Let $\mathcal{P}$ be a Markov transition operator
which is reversible with respect to $\mu$. If the following
relationship holds for some $f\in L^{2}(\mu)$, the constants $C(f)$,
and $ \lambda>0$
\[
\bigl\llVert \mathcal{P}^{n}f-\mu(f)\bigr\rrVert _{2}^{2}
\leq C(f)\exp (-\lambda n) \qquad \mbox{for all }n,
\]
then for the same $f$,
\[
\bigl\llVert \mathcal{P}^{n}f-\mu(f)\bigr\rrVert _{2}^{2}
\leq\bigl\llVert f-\mu(f)\bigr\rrVert _{2}^{2}\exp(-\lambda n)
\qquad \mbox{for all }n.
\]
\end{lem}
\begin{pf}
Without loss of generality we assume that $\mu(\hat{f}^{2})=1$ where
$\hat{f}=f-\mu(f)$.
Applying the spectral theorem to $\mathcal{P}$ yields the existence
of a unitary map $U\dvtx L^{2}(\mu)\mapsto L^{2}(X,\nu)$ such that $UPU^{-1}$
is a multiplication operator by $m$. Moreover, $\mu(\hat{f}^{2})=1$
implies that $(U\hat{f})^{2}\nu$ is a probability measure. Thus for
$k\in\mathbb{N}$,
\begin{eqnarray*}
\int\bigl(\mathcal{P}^{n}\hat{f}(x)\bigr)^{2}\,d\mu& = & \int
m(x)^{2n}(U\hat {f})^{2}(x)\,d\nu=\int m(x)^{(2n+k)\afrac{2n}{2n+k}}\,d(U
\hat{f})^{2}\nu
\\
& \leq& \biggl(\int m(x)^{2n+k}\,d(U\hat{f})^{2}\nu
\biggr)^{\afrac
{2n}{2n+k}}\leq C^{\afrac{2n}{2n+k}}\exp(-\lambda2n).
\end{eqnarray*}
Letting $k\rightarrow\infty$ yields the required claim.
\end{pf}

\subsection{Dimension-independent spectral gaps for the
pCN-algorithm}
\label{sub:Results}

Using the weak Harris theorem, we give necessary conditions on $\mu$
[see (\ref{eq:target})] in terms of regularity and growth of $\Phi$
to have a uniform spectral gap in a Wasserstein distance for $\mathcal{P}$
and $\mathcal{P}^{m}$. We need $\Phi$ to be at least locally Lipschitz;
the case where it is globally Lipschitz is more straightforward and
is presented first. Using the notation $\rho=1-(1-2\delta)^{\sfrac{1}{2}}$,
we can express the proposal of the pCN algorithm as
\[
p_{X^{n}}(\xi)=(1-\rho)X^{n}+\sqrt{2\delta} \xi.
\]
The following results do all hold for $\delta$ in $(0,\frac{1}{2}]$:

The mean of the proposal $(1-\rho)X^{n}$ suggests that we can prove
that $f(\llVert \cdot\rrVert )$ is a Lyapunov function for
certain $f$ and that $\mathcal{P}$ is $d$-contracting (for a suitable
metric). This relies on having a lower bound on the probability of $X_{n+1}$
being in a ball around the mean. In fact, our assumptions are stronger
because we assume a uniform lower bound on $\mathbb{P}(p_{x}\mbox{ is
accepted}|p_{x}=z)$
for $z$ in $B_{r(\llVert  x\rrVert )} ((1-\rho)x )$.
%
\begin{assumption}
\label{ass:acceptance}There is $R>0$, $\alpha_{l}>-\infty$
and a function $r\dvtx \mathbb{R}^{+}\mapsto\mathbb{R}^{+}$
satisfying $r(s)\le\frac{\rho}{2}s$ for all $\llvert s\rrvert \geq R$
such that for all $x\in B_{R}(0)^{c}$,
%
\begin{equation}
\label{eq:assLowerBdd} \hspace*{24pt}\inf_{z\in B_{r(\llVert  x\rrVert )}((1-\rho)x)}\alpha(x,z)=\inf_{z\in B_{r(\llVert  x\rrVert )}((1-\rho
)x)}
\exp \bigl(-\Phi(z)+\Phi(x) \bigr)>\exp(\alpha _{l}).
\end{equation}
\end{assumption}
%
\begin{assumption}
\label{ass:globalLip}Let $\Phi$ in (\ref{eq:target}) have global
Lipschitz constant $L$, and assume that $\exp(-\Phi)$ is $\gamma$-integrable.
\end{assumption}
%
\begin{theorem}
\label{thm:global} Let Assumptions \ref{ass:acceptance} and \ref
{ass:globalLip}
be satisfied with either:
\begin{longlist}[(2)]
\item[(1)] $r(\llVert  x\rrVert )=r\llVert  x\rrVert
^{a}$ where
$r\in\mathbb{R}^{+}$ for any $a\in(\frac{1}{2},1)$, and then we consider
$V=\llVert  x\rrVert ^{i}$ with $i\in\mathbb{N}$ or $V=\exp
(v\llVert  x\rrVert )$,
or
\item[(2)] $r(\llVert  x\rrVert )=r\in R$ for $r\in\mathbb
{R}^{+}$, and then
we take $V=\llVert  x\rrVert ^{i}$ with $i\in\mathbb{N}$.
\end{longlist}
Under these assumptions $\mu_m$ ($\mu$) is the unique invariant
measure for the Markov
chain associated with the pCN algorithm applied to $\mu_m$ ($\mu$).
Moreover, define
\begin{eqnarray*}
\tilde{d}(x,y) & = & \sqrt{d(x,y) \bigl(1+V(x)+V(y)\bigr)}\quad \mbox{with}
\\
d(x,y) & = & 1\wedge\frac{\llVert  x-y\rrVert }{\epsilon}.
\end{eqnarray*}
Then for $\epsilon$ small enough there exists an $\tilde{n}$ such that
for all probability measures $\nu_{1}$ and $\nu_{2}$ on $\mathcal{H}$
and $P_{m}\mathcal{H}$, respectively,
\begin{eqnarray*}
\tilde{d}\bigl(\nu_{1}\mathcal{P}^{\tilde{n}},\nu_{2}
\mathcal {P}^{\tilde{n}}\bigr) & \leq& \tfrac{1}{2}\tilde{d}(
\nu_{1},\nu_{2}),
\\
\tilde{d}\bigl(\nu_{1}\mathcal{P}_{m}^{\tilde{n}},
\nu_{2}\mathcal {P}_{m}^{\tilde{n}}\bigr) & \leq&
\tfrac{1}{2}\tilde{d}(\nu_{1},\nu_{2})
\end{eqnarray*}
for all $m\in\mathbb{N}$.
\end{theorem}
\begin{pf}
The conditions of the weak Harris theorem (Proposition~\ref{thm:weakHarris})
are satisfied by the Lemmata \ref{lem:lyapunov}, \ref{lem:globalContraction}
and \ref{lem:globealLSmallness}.
\end{pf}

A key step in the proof is to verify the $d$-contraction. In order
to obtain an upper bound on $d(\mathcal{P}(x,\cdot),\mathcal
{P}(y,\cdot))$
[see (\ref{eq:liftedMetric})], we choose a particular coupling between
the algorithm started at $x$ and $y$ and distinguish between the
cases when both proposals are accepted, both are rejected, and only
one is accepted. The case when only one of them is accepted is the most
difficult to tackle. By choosing $d=1\wedge\frac{\llVert
x-y\rrVert }{\epsilon}$
with $\epsilon$ small enough, it turns out that the Lipschitz constant
of $\alpha(x,y)$ can be brought under control.

By changing the distance function $d$, we can also handle the case when
$\Phi$ is locally Lipschitz provided that the local Lipschitz constant
does not
grow too fast.
%
\begin{assumption}
\label{ass:LocalLipschitz}Let $\exp(-\Phi)$ be integrable with respect
to $\gamma$, and assume that for any $\kappa>0$, there is an
$M_{\kappa}$
such that
\[
\phi(r)=\sup_{x\neq y\in B_{r}(0)}\frac{\llvert \Phi
(x)-\Phi(y)\rrvert }{\llVert  x-y\rrVert }\leq M_{\kappa
}e^{\kappa r}.
\]
\end{assumption}
%
\begin{theorem}
\label{thm:localVaryingBalls}Let the Assumptions \ref{ass:acceptance}
and \ref{ass:LocalLipschitz} be satisfied with\linebreak[4]  $r(\llVert  x\rrVert )=r\llVert  x\rrVert ^{a}$
where $r\in\mathbb{R}, a\in(\frac{1}{2},1)$ and either $V=\llVert  x\rrVert ^{i}$
with $i\in\mathbb{N}$ or $V=\exp(v\llVert  x\rrVert )$.

Then $\mu_m$ ($\mu$) is the unique invariant measure for the Markov
chain associated with the pCN algorithm applied to $\mu_m$ ($\mu$).

For $\mathsf{A}(T,x,y):=\{\psi\in C^{1}([0,T],\mathcal{H}),\psi
(0)=x,\psi(T)=y,\|\dot{\psi}\|=1\}$,
\begin{eqnarray*}
\tilde{d}(x,y) & = & \sqrt{d(x,y) \bigl(1+V(x)+V(y)\bigr)}\quad \mbox{ with}
\\
d(x,y) & = & 1\wedge\inf_{T,\psi\in\mathsf{A}(T,x,y)}\frac{1}{\epsilon}\int
_{0}^{T}\exp\bigl(\eta\llVert \psi\rrVert \bigr)\,dt
\end{eqnarray*}
and $\eta$ and $\epsilon$ small enough there exists an $\tilde{n}$ such
that for all $\nu_{1},\nu_{2}$ probability measures on $\mathcal{H}$
and on $P_{m}\mathcal{H}$, respectively, and $m\in\mathbb{N}$
\begin{eqnarray*}
\tilde{d}\bigl(\nu_{1}\mathcal{P}^{\tilde{n}},\nu_{2}
\mathcal {P}^{\tilde{n}}\bigr) & \leq& \tfrac{1}{2}\tilde{d}(
\nu_{1},\nu_{2}),
\\
\tilde{d}\bigl(\nu_{1}\mathcal{P}_{m}^{\tilde{n}},
\nu_{2}\mathcal {P}_{m}^{\tilde{n}}\bigr) & \leq&
\tfrac{1}{2}\tilde{d}(\nu_{1},\nu_{2}).
\end{eqnarray*}
\end{theorem}
%
%
\begin{pf}
This time Lemmata \ref{lem:lyapunov}, \ref{lem:localContraction}
and \ref{lem:localSmallness} verify the conditions of the weak Harris
theorem (Proposition~\ref{thm:weakHarris}).
\end{pf}
\begin{rem*}
Our arguments work for $\delta\in(0,\frac{1}{2}]$; for $\delta
=\frac{1}{2}$,
the pCN algorithm becomes the independence sampler, and the Markov
transition kernel becomes irreducible so that this case we can use the
theory of
\citet{Meyn:2009uqa}.
\end{rem*}
In order to get the same lower bound for the $L_{\mu}^{2}$-spectral
gap, we just\vspace*{-1pt} have to verify that $\operatorname{Lip}(\tilde{\delta})\cap L_{\mu
}^{\infty}$
is dense in $L_{\mu}^{2}$.
%
\begin{theorem}
\label{thm:wassersteinImplL2}If the conditions of Theorem~\ref{thm:global}
or \ref{thm:localVaryingBalls} are satisfied, then we have the same
lower bound on the $L_{\mu}^{2}$-spectral gap of $\mathcal{P}$ and
$\mathcal{P}_{m}$ uniformly in~$m$.
\end{theorem}
\begin{pf}
By Proposition~\ref{prop:wassersteinL2} we only have to show that
$\operatorname{Lip}(\tilde{d})\cap L^{\infty}(\mu)$ is dense in $L^{2}(H,\mathcal
{B},\mu)$.
Since $\tilde d(x,y) \ge C(1 \wedge\|x-y\|)$, one has
$\operatorname{Lip}(\llVert \cdot\rrVert )\cap L^\infty(\mu) \subseteq
\operatorname{Lip}(\tilde{d})$,
so that it is enough to show that $\operatorname{Lip}(\llVert \cdot\rrVert
)\cap L^{\infty}(\mu)$
is dense in $L^{2}(H,\mathcal{B},\mu)$. Suppose not; then there is
$0\neq g\in L^{2}(\mu)$ such that
\[
\int fg\,d\mu=0 \qquad \mbox{for all }f\in \operatorname{Lip}\cap \,L^{\infty}(\mu).
\]
Since all Borel probability measures on a separable Banach space
are characterised by their Fourier transform
[see, e.g., \citet{Bogachev:2006ys}], they
are characterised by integrals against bounded Lipschitz functions.
Hence $g\,d\mu$ is the zero measure and hence $g\equiv0$ in $L_{\mu}^{2}$.
\end{pf}

\subsection{Dimension-dependent spectral gaps for RWM}
\label{sub:RWM}

So far we have shown convergence results for the pCN. Therefore
we present an example subsequently where these results apply but the
spectral gap
of the RWM goes to $0$ as $m$ tends to infinity. We consider the
target measures $\mu_m$ on
\[
\mathcal{H}_{m}^{\sigma}:= \Biggl\{ x \Bigl| \llVert x\rrVert
_{\sigma}=\sum_{i=1}^{m}i^{2\sigma}x_{i}^{2}<
\infty \Biggr\}
\]
with $0<\sigma<\frac{1}{2}$ given by
%
\begin{equation}\label{eq:targetCounterexample}
\mu_{m}=\gamma_{m}=\mathcal{L} \Biggl(\sum
_{i=1}^{m}\frac{1}{i}\xi _{i}e_{i}
\Biggr),\qquad  \xi\stackrel{\mathrm{i.i.d.}} {\sim}\mathcal {N}(0,1).
\end{equation}
In the setting of (\ref{eq:target}) this corresponds to $\Phi=0$.
Hence the assumptions of Theorem~\ref{thm:localVaryingBalls} are
satisfied, and we obtain a uniform lower bound on the $L_{\mu}^{2}$-spectral
gap for the pCN. For the RWM algorithm we show that the spectral gap
converges to zero faster than any negative power of $m$ if we scale
$\delta=s m^{-a}$ for any $a\in[0,1)$.

Using the notion of conductance,
%
\begin{equation}
\label{def:conductance} \mathsf{C}=\inf_{\mu(A)\leq\sfrac{1}{2}}\frac{\int_{A}\mathcal{P}(x,A^{c})\,d\mu(x)}{\mu(A)},
\end{equation}
we obtain an upper bound on the spectral gap by Cheeger's inequality
[\citet{lawler1988bounds,sinclair1989approximate}],
%
\begin{equation}
\label{eq:cheeger} 1-\beta\leq2\mathsf{C}.
\end{equation}

Our main observation is that there is a simple upper bound for the
conductance of a Metropolis--Hastings algorithm because it can only
move from a set $A$ if:
\begin{itemize}
\item the proposed move lies in $A^{c}$, and
\item the proposed move is accepted.
\end{itemize}
Just considering either event gives rise to simple upper bounds that
can be used to make many results from the scaling analysis rigorous.
We denote the expected acceptance probability for a proposal from
$x$ as
\[
\alpha(x)=\int_{\mathcal{H}}\alpha(x,y)\,dQ(x,dy).
\]
Considering only the acceptance of the proposal gives rise to
\[
C\leq\inf_{\mu(A)\leq\sfrac{1}{2}}\frac{\int_{A}\alpha
(x)\mu(dx)}{\mu(A)}.
\]
In particular, for any set $B$ such that $\mu(B)\leq\frac{1}{2}$,
it follows that
\[
C\leq\sup_{x\in B}\alpha(x)
\]
and also that
\[
C\leq2\mathbb{E}_{\mu}\alpha(x).
\]
The last result allows us to make scaling results like those in \citet{mcmcLocalScaling}
rigorous. Similarly, just supposing that the Metropolis--Hastings algorithm
accepts all proposals gives rise to the following bound:
\[
C\leq\inf_{\mu(A)\leq\sfrac{1}{2}}\frac{\int_{A}Q(x,A^{c})\,d\mu(x)}{\mu(A)}.
\]
We summarise these observations in the subsequent proposition.
%
\begin{prop} \label{thm:upperboundMetro} Let
$\mathcal{P}$ be a Metropolis--Hastings transition kernel for a target
measure $\mu$ with proposal kernel $Q(x,dy)$ and acceptance probability
$\alpha(x,y)$. The $L_{\mu}^{2}$-spectral gap can be bounded
by
%
\begin{equation}
1-\beta\leq1-\Lambda\leq2C\leq2 %
\cases{ \displaystyle \sup_{x\in B}
\alpha(x), &\quad  \mbox{for any} $\mu(B)\leq\displaystyle \frac{1}{2}$,
 \vspace*{2pt}\cr
\mathbb{E}_{\mu}
\alpha(x), } %
\label{eq:MCMCGapConAcc}
\end{equation}
and
%
\begin{equation}\label
{eq:MCMCGapConProp}
1-\beta\leq1-\Lambda\leq2C\leq2\inf_{\mu(A)\leq\sfrac
{1}{2}}
\frac{\int_{A}Q(x,A^{c})\,d\mu(x)}{\mu(A)}.
\end{equation}
\end{prop}

In the following theorem we use the Proposition~\ref
{thm:upperboundMetro} for the RWM algorithm applied to $\mu_{m}$ as in equation
(\ref{eq:targetCounterexample}) in order to quantify the behaviour of
the spectral gap as $m$ goes to $\infty$.
We consider polynomial scaling of the step size parameter of the form
$\delta_{m}\sim m^{-a}$ to zero. For $a<1$ the bound
in equation \eqref{eq:MCMCGapConAcc} is most useful as the acceptance
behaviour is the determining quantity. For $a\ge1$ the bound
in equation \eqref{eq:MCMCGapConProp} is most useful as the properties
of the proposal kernel are determining in this regime.
%
\begin{theorem}
\label{thm:negativeRWM}Let $\mathcal{P}_{m}$ be the Markov kernel
and $\alpha$ be the acceptance probability associated with the RWM
algorithm applied to $\mu_{m}$ as in equation~(\ref{eq:targetCounterexample}).
\begin{longlist}[(2)]
\item[(1)] For $\delta_{m}\sim m^{-a}$, $a\in[0,1)$ and any $p$
there exists a
$K(p,a)$ such that the spectral gap of $\mathcal{P}_{m}$ satisfies
\[
1-\beta_{m}\leq K(p,a)m^{-p}.\vadjust{\goodbreak}
\]

\item[(2)] For $\delta_{m}\sim m^{-a}$, $a\in[1,\infty)$ there
exists a $K(a)$
such that the spectral gap of $\mathcal{P}_{m}$ satisfies
\[
1-\beta_{m}\leq K(a)m^{-\sfrac{a}{2}}.
\]
\end{longlist}
\end{theorem}
\begin{pf}
For the first part of this proof we work on the space $H_{\sigma}$
with $\sigma\in[0,\frac{1}{2})$
where we determine $\sigma$ later. We choose $B_{r}(0)$ such that
$\mu (B_{r}(0) )\leq\frac{1}{4}$ and by (\ref
{lem:approximatedMeasure}) below
we know that $\mu_{m} (B_{r}^{m}(0) )$ is decreasing toward
$\mu (B_{r}(0) )$. Hence for all $m$ larger than some $M$
we know that $\mu (B_{r}^{m}(0) )\leq\frac{1}{2}$. In order
to apply Proposition~\ref{thm:upperboundMetro}, we have to gain an
upper bound on $\alpha(x)$ in $B_{r}^{m}(0)$. Thus we use $u\wedge
v\leq u^{\lambda}v^{1-\lambda}$
to bound
\[
\alpha(x,y)=1\wedge\exp \Biggl(-\sum_{i=1}^{m}
\frac
{i^{2}}{2}\bigl(y_{i}^{2}-x_{i}^{2}
\bigr) \Biggr)\leq\exp \Biggl(-\sum_{i=1}^{m}
\frac{i^2}{2}\bigl(y_{i}^{2}-x_{i}^{2}
\bigr)\lambda \Biggr).
\]
Using this inequality, we can find an upper bound on the acceptance
probability $\alpha(x)$.
\[
\int\alpha Q(x,dy)\leq\int\frac{m!}{(4\delta\pi)^{\sfrac
{m}{2}}}\exp \Biggl(-\sum
_{i=1}^{m}\frac{i^2}{2} \biggl[
\bigl(y_{i}^{2}-x_{i}^{2}\bigr)\lambda+
\frac{(x_{i}-y_{i})^{2}}{2\delta
} \biggr] \Biggr)\,dy.
\]
Completing the square and using the normalisation constant yields
\begin{eqnarray*}
&\leq& \int\frac{m!}{(4\delta\pi)^{\sfrac{m}{2}}}\exp \Biggl(-\sum_{i=1}^{m}
\frac{i^2}{2} \biggl[\biggl(\lambda+\frac{1}{2\delta}\biggr)
\biggl(y_{i}-\frac{x_{i}}{2\delta\lambda+1} \biggr)^{2}-
\frac{2\delta
\lambda^{2}x_{i}^{2}}{(2\delta\lambda+1)} \biggr] \Biggr)\,dy
\\
&\leq& (1+2\lambda\delta)^{-\sfrac{m}{2}}\exp \Biggl(\sum
_{i=1}^{m}\frac{\delta\lambda^{2}i^{2}x_{i}^{2}}{(2\delta\lambda
+1)} \Biggr).
\end{eqnarray*}
For $x\in B_{r}^{m}(0)$ in $\mathcal{H}_{\sigma}$, using $\delta=m^{-a}$
and setting $\lambda=m^{-b}$
\[
\alpha(x)  \leq \bigl(1+2m^{-(a+b)}\bigr)^{-\sfrac{m}{2}}\exp \biggl(
\frac
{rm^{2-2\sigma-a-2b}}{3} \biggr).
\]
We want to choose $a$ and $b$ in the above equation such that the right-hand side
goes to zero as $m\rightarrow\infty$. In order to obtain decay from
the first factor, we need that $a+b<1$ and to
prevent growth from the second $a+2b>2-2\sigma$ which corresponds to
$a+2b>1$ for $\sigma$ sufficiently close to $\frac{1}{2}$. This
can be satisfied with $b=\frac{2(1-a)}{3}$ and $\sigma=\frac
{2+a}{6}<\frac{1}{2}$.
In this case the first factor decays faster than any negative power
of $m$ since
\[
\bigl(1+2m^{-(a+b)}\bigr)^{-\sfrac{m}{2}}=\exp \biggl(-\frac{m}{2}
\log \bigl(1+2m^{-(a+b)}\bigr) \biggr)\leq\exp\bigl(-Cm^{1-(a+b)}
\bigr).
\]

For the second part of the poof we use $\alpha(x,y)\leq1$ and $A=\{
x\in\mathbb{R}^{n}\vert x_{1}\geq0\}$
which by using a symmetry argument satisfies $\gamma_{m}(A)=\frac
{1}{2}$ to bound
the conductance
\begin{eqnarray*}
\frac{\mathsf{C}}{2} & \leq& \int_{A}P\bigl(x,A^{c}
\bigr)\,d\mu
\\
&\leq&\int_{A}\int_{A^{c}}
\frac{\alpha
(x,y)n!^{2}}{(2\pi)^{n}(2\delta)^{\sfrac{n}{2}}}\exp \Biggl(-\frac
{1}{2}\sum_{i=1}^{m}i^{2}
\bigl(x_{i}^{2}+(x_{i}-y_{i})^{2}/(2
\delta) \bigr)\Biggr)\,dx\,dy
\\
& \leq&\int_{0}^{\infty}\int_{-\infty}^{0}
\exp\biggl(-\frac{1}{2}\frac
{(y_{1}-x_{1})^{2}}{2\delta}\biggr)\bigl/{(2\pi\sqrt{2\delta})}\,dy_{1}\exp \biggl(-\frac{1}{2}x_{1}^{2}
\biggr)\,dx_{1}
\\
& = & \int_{0}^{\infty}\int_{-\infty}^{-\sfrac{x_{1}}{\sqrt
{2\delta}}}
\exp\biggl(-\frac{1}{2}z^{2}\biggr)\bigl/{(2\pi)}\,dy_{1}\exp \biggl(-\frac{1}{2}x_{1}^{2}
\biggr)\,dx_{1}.
\end{eqnarray*}
Combining Fernique's theorem and Markov's inequality
yields
\[
\mathsf{C}\leq K\int_{0}^{\infty}\exp\biggl(-
\frac{1}{2}\biggl(\frac{\delta
+1}{\delta}\biggr)x_{1}^{2}
\biggr)\,dx\leq K\sqrt{2\pi\frac{\delta}{\delta
+1}}\leq\tilde{K}m^{-\sfrac{a}{2}},
\]
so that the claim follows again by an application of Cheeger's inequality.
\end{pf}

\section{Spectral gap: Proofs}\label{proof for spectral gaps}

We check the three conditions of the weak Harris theorem (Proposition~\ref{thm:weakHarris}) for globally and locally Lipschitz $\Phi$
[see (\ref{eq:target})] in the Sections~\ref{sub:proofWassersteinGlobal}
and \ref{sub:proofWassersteinLocal}, respectively. For each condition
we use the following lemma for the dependence of the constants $l$, $K$, $c$
and $s$ in the weak Harris theorem on~$m$. This allows us to conclude
that there exists $\tilde{n}(m)\leq\tilde{n}$ such that
\begin{eqnarray*}
\tilde{d}\bigl(\nu_{1}\mathcal{P}^{\tilde{n}},\nu_{2}
\mathcal {P}^{\tilde{n}}\bigr) & \leq& \tfrac{1}{2}\tilde{d}(
\nu_{1},\nu_{2}),
\\
\tilde{d}\bigl(\nu_{1}\mathcal{P}_{m}^{\tilde{n}(m)},
\nu_{2}\mathcal {P}_{m}^{\tilde{n}(m)}\bigr) & \leq&
\tfrac{1}{2}\tilde{d}(\nu_{1},\nu_{2})
\end{eqnarray*}
for all measures $\nu_{1}, \nu_{2}$ probability measures on $\mathcal
{H}$ and $P_{m}\mathcal{H}$,
respectively.

Replacing $r(s)\wedge\frac{\rho}{2}s$ only weakens the condition
(\ref{eq:assLowerBdd}), so we can and will assume that $r(s)\leq\rho s/2$.
%
\begin{lem}
\label{lem:approximatedMeasure}Let $f\dvtx \mathbb{R}\rightarrow\mathbb{R}$
be monotone increasing, then
\[
\int f\bigl(\llVert \xi\rrVert \bigr)\,d\gamma_{m}(\xi)\leq\int f\bigl(
\llVert \xi\rrVert \bigr)\,d\gamma(\xi),
\]
and in particular
%
\begin{equation}\label{eq:approximation of Balls}
\gamma_{m}\bigl(B_{R}(0)\bigr)\geq\gamma
\bigl(B_{R}(0)\bigr).
\end{equation}
\end{lem}
\begin{pf}
The truncated Karhunen--Lo\`{e}ve expansion relates $\gamma_{m}$ to
$\gamma$
and yields
\[
\sum_{i=1}^{m}\lambda_{i}
\xi_{i}^{2}\leq\sum_{i=1}^{\infty
}
\lambda_{i}\xi_{i}^{2}.
\]
Hence the result follows by monotonicity of the integral and of the
function~$f$
\begin{eqnarray*}
\int f\bigl(\llVert \xi\rrVert \bigr)\,d\gamma_{m}(\xi)&=&\mathbb {E}
\Biggl(\sqrt{f\Biggl(\sum_{i=1}^{m}
\lambda_{i}\xi_{i}^{2}\Biggr)}\Biggr)\leq\mathbb
{E}\Biggl(\sqrt{f\Biggl(\sum_{i=1}^{\infty}
\lambda_{i}\xi_{i}^{2}\Biggr)}\Biggr)\\
&=&\int f\bigl(
\llVert \xi\rrVert \bigr)\,d\gamma(\xi).
\end{eqnarray*}
This yields equation (\ref{eq:approximation of Balls}) by inserting
$f=\chi_{B_{R}(0)^{c}}$.
\end{pf}

\subsection{Global log-Lipschitz density}
\label{sub:proofWassersteinGlobal}

In this section we will prove Theorem~\ref{thm:global} by checking
the three conditions of the weak Harris Theorem~\ref{thm:weakHarris}
for the distance-like functions
%
\begin{equation}\label
{eq:globaldistance}
d(x,y)=1\wedge\frac{\llVert  x-y\rrVert }{\epsilon}.
\end{equation}

\subsubsection{Lyapunov functions}

Under Assumption~\ref{ass:acceptance} we show the existence of a
Lyapunov function $V$. This follows from two facts: First, the decay of
$V$ on $B_{r(\llVert  x\rrVert )} ((1-\rho)x )$
and second the probability of the next step of the algorithm lying
in that ball can be bounded below by Fernique's theorem;
see Proposition~\ref{pro:splitMomentsEsitmates}.
Similarly, we will use the second part of Proposition~\ref
{pro:splitMomentsEsitmates} to deal with proposals outside $B_{r(\llVert  x\rrVert )} ((1-\rho)x )$.
%
\begin{lem}
\label{lem:lyapunov}If Assumption~\ref{ass:acceptance} is satisfied
with:
\begin{longlist}[(2)]
\item[(1)] $r(\llVert  x\rrVert )=r\in\mathbb{R}$ or
\item[(2)] $r(\|x\|)=r\|x\|^{a}$, $\kappa>0$ and $a\in(\frac{1}{2},1)$,
\end{longlist}
then the function $V(x)=\llVert  x\rrVert ^{i}$ with $i\in
\mathbb{N}$
in the first case and additionally $V(x)=\exp (\ell\|x\| )$
in the second case are Lyapunov functions for both $\mathcal{P}$
and $\mathcal{P}_{m}$ with constants $l$ and $K$ uniform in m.
\end{lem}
\begin{pf}
In both cases we choose $R$ as in Assumption~\ref{ass:acceptance}.
Then there exists a constant $K_1$ such that
\[
\sup_{x\in B_{R}(0)}\mathcal{P}V(x)\leq\sup_{x\in
B_{R}(0)}
\int \bigl(\llVert x\rrVert +\sqrt{2\delta }\llVert \xi\rrVert
\bigr)^{i}\,d\gamma(\xi) =: K_{1}<\infty.
\]
On the other hand, there exists
$0<\tilde{l}<1$ such that for all $x\in B_{R}(0)^{c}$,
%
\begin{equation}\label{eq:LyapunovFctCondition}
\sup_{y\in B_{r(\llVert  x\rrVert )}((1-\rho)x)} V(y)\leq\tilde{l}V(x).
\end{equation}
We denote by $A=\{\omega|\sqrt{2\delta}\llVert \xi\rrVert
\leq r(\llVert  x\rrVert )\}$
the event that the proposal lies in a ball with a lower bound on the acceptance
probability due to Assumption~\ref{ass:acceptance}. This yields the bound
\begin{eqnarray*}
\mathcal{P}V & \leq& \mathbb{P}(A) \bigl[\mathbb{P}(\mathrm{accept}|A)\tilde {l}V(x)+
\mathbb{P}(\mathrm{reject}|A)V(x) \bigr]+\mathbb{E}\bigl(V(p_{x})\vee
V(x);A^{c}\bigr)
\\
& \leq& \mathbb{P}(A) \bigl[\bigl(1-\mathbb{P}(\mathrm{accept}|A) (1-\tilde {l})\bigr) \bigr]V(x)+
\mathbb{E}\bigl(V(p_{x})\vee V(x);A^{c}\bigr)
\\
& \leq& \theta\mathbb{P}(A)V(x)+\mathbb{E}\bigl(V(p_{x})\vee
V(x);A^{c}\bigr)
\end{eqnarray*}
for some $\theta<1$. It remains to consider $\mathbb{E}(V(p_{x})\vee
V(x);A^{c})$
where the differences will arise between the cases 1 and 2. For the first
case we know that by an application of Fernique's theorem
\begin{eqnarray*}
\mathbb{E}\bigl(V(p_{x})\vee V(x);A^{c}\bigr) & \leq& \int
_{\sqrt{2\delta
}\llVert \xi\rrVert \geq c}\llVert x\rrVert ^{i}\vee \bigl((1-\rho)
\llVert x\rrVert +\sqrt{2\delta}\llVert \xi\rrVert \bigr){}^{i}\,d
\gamma(\xi)
\\
& \leq& \int_{\llVert \xi\rrVert \geq\sfrac{c}{\sqrt
{2\delta}}} \bigl(\llVert x\rrVert ^{i}+K
\llVert \xi \rrVert ^{p} \bigr)\,d\gamma(\xi)
\\
&\leq&\mathbb{P}\bigl(A^{c}\bigr)V(x)+K_{2}.
\end{eqnarray*}

Because a ball around the mean of a Gaussian measure on a separable space
always has positive mass
[Theorem~3.6.1 in \citet{gaussianMeasureas}], we note that
\[
\mathcal{P}V(x)\leq V(x) \bigl(\mathbb{P}(A)\theta+\mathbb {P}
\bigl(A^{c}\bigr)\bigr)+K_{2}\leq lV(x)+K_{2}
\]
for some constant $l < 1$.

For the second case we estimate
\begin{eqnarray*}
\mathbb{E}\bigl(V(p_{x})\vee V(x);A^{c}\bigr) & \leq&
M_{v}\int_{\llVert
\eta\rrVert >r\llVert  x\rrVert ^{a}}e^{v(\llVert
x\rrVert +\sqrt{2\delta}\llVert \xi\rrVert )}\,d\gamma (\xi).
\end{eqnarray*}
The right-hand side of the above is uniformly bounded in $x\in B_{R}(0)^{c}$
by some $K_{2}$ due to Proposition~\ref{pro:splitMomentsEsitmates}.
Hence in both cases there exists an $l<1$ such that
\[
\mathcal{P}V(x)\leq lV(x)+\max(K_{1},K_{2}) \qquad \forall x.
\]

For the $m$-dimensional approximation $\mathcal{P}_m$ the probability
of the event
$A$ is larger than $\mathcal{P}$ by Lemma~\ref
{lem:approximatedMeasure}. Since there is a common lower bound for
$\mathbb{P}(\mathrm{accept}|A)$ $l(m)$ is smaller than or equal to $l$.
Similarly, $K_{i}(m)$ is smaller than $K_{i}$ by Lemma~\ref
{lem:approximatedMeasure}.
\end{pf}

\subsubsection{The $d$-contraction}
\label{sub:basicCoupling}
In this section we show that $\mathcal{P}$ is $d$-contracting for
$d(x,y)=1\wedge\frac{\llVert  x-y\rrVert }{\epsilon}$ by bounding
$d(\mathcal{P}(x,\cdot),\mathcal{P}(y,\cdot))$ [see (\ref
{eq:liftedMetric})]
with a particular coupling. For $x$ and $y$ we choose the same noise
$\xi$ giving rise to the proposals $p_{x}(\xi)$ and $p_{y}(\xi)$
and the same uniform random variable for acceptance. The situation
is illustrated in Figure \ref{fig:GlobalContraction}.
Subsequently,
we will refer to this coupling as the basic coupling and bound the
expectation of $d$ under this coupling by inspecting the following
cases:
\begin{longlist}[(3)]
\item[(1)] the proposals for the algorithm started at $x$ and $y$ are both
accepted;
\item[(2)] both proposals are rejected;
\item[(3)] one of the proposals is accepted and the other rejected.
\end{longlist}

%
\begin{lem}
\label{lem:globalContraction}If $\Phi$ in (\ref{eq:target}) satisfies
Assumptions \ref{ass:acceptance} and \ref{ass:globalLip}, then
$\mathcal{P}$
and $\mathcal{P}_{m}$ are $d$-contracting for $d$ as in (\ref
{eq:globaldistance})
with a contraction constant uniform in $m$.
\end{lem}

\begin{pf}
By Definition~\ref{def:d-contracting} we only need to consider $x$
and $y$ such that $d(x,y)<1$ which implies that $\llVert  x-y\rrVert <\epsilon$.
Later we will choose $\epsilon\ll1$ so that if $\llVert  x-y\rrVert <\epsilon$,
then either $x,y\in B_{R}(0)$ or $x,y\in B_{\tilde{R}}^{c}(0)$ with
$\tilde{R}=R-1$, and we will treat both cases separately. We assume
without loss of generality that $\llVert  y\rrVert \geq\llVert  x\rrVert $.

%
\begin{figure}

\includegraphics{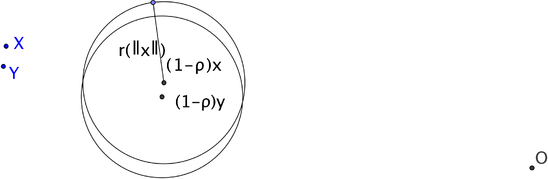}

\caption{Contraction.}
\label{fig:GlobalContraction}
\end{figure}

For $x,y\in B_{R}(0)$ and $A=\{\omega|\sqrt{2\delta}\llVert \xi
\rrVert \leq R\}$,
the basic coupling yields
%
\begin{eqnarray}\label{eq:contrBdd}
\qquad &&d\bigl(\mathcal{P}(x,\cdot),\mathcal{P}(y,\cdot)\bigr) \nonumber
\\
&&\qquad  \leq \mathbb {P}(A)
\bigl[\mathbb{P}(\mathrm{both\ accept}|A) (1-\rho)d(x,y)
+\mathbb{P} (\mathrm{both\ reject}|A)d(x,y) \bigr]
\\
&&\quad\qquad{} +\mathbb{P}
\bigl(A^{c}\bigr)d(x,y)
+\int_{\mathbf{\mathcal{H}}}\bigl\llvert \alpha(x,p_{x}) (\xi)-
\alpha (y,p_{y}) (\xi)\bigr\rrvert \,d\gamma(\xi),\nonumber
\end{eqnarray}
where the last term bounds the case that only one of the proposals
is accepted. Using the bound $\mathbb{P}(\mathrm{both\ reject|A})\leq
1-\mathbb{P}(\mathrm{both\ accept}|A)$
yields a nontrivial convex combination of $d$ and $(1-\rho)d$
because the probability $\mathbb{P}(\mathrm{both\ accept}|A)$ is bounded
below by $\exp(-\sup\{\Phi(z)\vert\llVert  z\rrVert \leq
2R\}+\break \inf\{\Phi(z)\vert\llVert  z\rrVert \leq2R\})$
due to (\ref{eq:defPRWM}). The first two summands in (\ref{eq:contrBdd})
form again a nontrivial convex combination since $\mathbb{P}(A)>0$
so that there is $\tilde{c}<1$ with\vspace*{-1pt}
\[
d\bigl(\mathcal{P}(x,\cdot),\mathcal{P}(y,\cdot)\bigr)\leq\tilde {c}d(x,y)+\int
_{\mathbf{\mathcal{H}}}\bigl\llvert \alpha(x,p_{x}) (\xi )-
\alpha(y,p_{y}) (\xi)\bigr\rrvert \,d\gamma(\xi).
\]
Note that $\tilde{c}$ is independent of $\epsilon$. For the last
term we use that $1\wedge\exp(\cdot)$ has Lipschitz constant $1$,
%
\begin{eqnarray}\label{GlobLipPhi}
&&\int_{X}\bigl|\alpha(x,p_{x}) (\xi)-
\alpha(y,p_{y}) (\xi)\bigr|\,d\gamma(\xi)
\nonumber\\[-1pt]
&&\qquad  \leq\int_{\mathcal{H}}\bigl\llvert \Phi(p_{x})-
\Phi(p_{y})\bigr\rrvert +\bigl\llvert \Phi(x)-\Phi(y)\bigr\rrvert\, d
\gamma(\xi)
\\[-1pt]
&&\qquad \leq2L\llvert x-y\rrvert \leq2L\epsilon d(x,y)\nonumber  %
\end{eqnarray}
which yields an overall contraction for $\epsilon$ small enough.\vspace*{1pt}

Similarly, we get for $x,y\in B_{\tilde{R}}(0)^{c}$ and $B=\{\omega
|\sqrt{2\delta}\llVert \zeta\rrVert \leq r(\llVert
x\rrVert \wedge\llVert  y\rrVert )\}$
\begin{eqnarray*}
d\bigl(\mathcal{P}(x,\cdot),\mathcal{P}(y,\cdot)\bigr) &\leq&\mathbb {P}(B)\bigl[
\mathbb{P}(\mathrm{both\ accept}|B) (1-\rho)+\mathbb{P}(\mathrm{both\ reject}|B)
\bigr]d(x,y)
\\
&& {}+\mathbb{P}\bigl(B^{c}\bigr)d(x,y)+\int_{\mathbf{\mathcal{H}}}
\bigl\llvert \alpha (x,p_{x}) (\xi)-\alpha(y,p_{y}) (\xi)
\bigr\rrvert \,d\gamma(\xi).
\end{eqnarray*}
The lower bound for $\mathbb{P}(\mathrm{both\ accept}|B)$ follows this time
from Assumption~\ref{ass:acceptance}.

All occurring ball probabilities are larger in the $m$-dimensional
approximation due to Lemma~\ref{lem:approximatedMeasure}, and the
acceptance probability is larger since $\inf$ and $\sup$ are applied
to smaller sets. Thus the contraction constant is uniform in $m$.
\end{pf}

\subsubsection{The $d$-smallness}

The $d$-smallness of the level sets of $V$ is achieved by replacing
the Markov kernel by the $n$-step one. This preserves the $d$-contraction
and the Lyapunov function. The variable $n$ is chosen large enough
so that if the algorithms started at $x$ and $y$ both accept $n$
times in a row, then $d$ drops below $\frac{1}{2}$. Hence
\[
d\bigl(\mathcal{P}^{n}(x,\cdot),\mathcal{P}^{n}(y,\cdot)
\bigr)\leq1-\tfrac
{1}{2}\mathbb{P}(\mbox{accept $n$-times}).
\]

\begin{lem}
\label{lem:globealLSmallness}If $S$ is bounded, then there exists an
$n$ and $0<s<1$ such that for all $x,y\in S$, $m\in\mathbb{N}$
and for $d$ as in (\ref{eq:globaldistance}),
\[
d \bigl(\mathcal{P}_{m}^{n}(x,
\cdot),\mathcal{P}_{m}^{n}(y,\cdot ) \bigr)\leq s \quad  \mbox{and}
\quad  d \bigl(\mathcal{P}^{n}(x,\cdot ),\mathcal{P}^{n}(y,\cdot)
\bigr)\leq s.  %
\]
\end{lem}
\begin{pf}
In order to obtain an upper bound for $d (\mathcal{P}^{n}(x,\cdot
),\mathcal{P}^{n}(y,\cdot) )$,
we choose the basic coupling (see Section~\ref{sub:basicCoupling})
as before. Let $R_{S}$ be such that $S\subset B_{R_{S}}(0)$ and
$B$ be the event that both instances of the algorithm accept $n$
times in a row. In the event of $B$, it follows by the definition of
$d$ [cf. (\ref{eq:globaldistance})] that
\[
d(X_{n},Y_{n})\leq\frac{1}{\epsilon}\llVert
X_{n}-Y_{n}\rrVert \leq\frac{1}{\epsilon}(1-
\rho)^{n}\llVert X_{0}-Y_{0}\rrVert \leq
\frac{1}{\epsilon}(1-\rho)^{n}\diam S\leq\frac{1}{2}
\]
which implies that if $X_{0}$ and $Y_{0}$ are in $S$, then
$d(X_{n},Y_{n})\leq\frac{1}{2}$.
Hence
\[
d\bigl(\mathcal{P}^{n}(x,\cdot),\mathcal{P}^{n}(y,\cdot)
\bigr)\leq\mathbb {P}(B)\tfrac{1}{2}+\bigl(1-\mathbb{P}(B)\bigr)\cdot1<1.
\]
Writing $\xi^{i}$ for the noise in the $i$th step, we bound
\begin{eqnarray*}
\mathbb{P}(B) & \geq&\mathbb{P} \biggl(\bigl\llVert \sqrt{2\delta}\xi
^{i}\bigr\rrVert \leq\frac{R}{n} \mbox{ for $i=1,\dots, n$}
\biggr)\mathbb{P} \biggl(\mathrm{both\ accept\ }n\mbox{ times} \Bigl| \bigl\llVert
\xi^{i}\bigr\rrVert \leq\frac{R}{n} \biggr)
\\
& \geq&\mathbb{P} \biggl(\llVert \zeta\rrVert \leq\frac
{R}{n}
\biggr)^{n}\exp \Bigl(-\sup_{z\in B_{2R}(0)}\Phi(z)+\inf
_{z\in B_{2R}(0)}\Phi(z) \Bigr)^{n}>0,
\end{eqnarray*}
uniformly for all $X_{0},Y_{0}\in B_{R}(0)$. For the $m$-dimensional
approximation the lower bound exceeds that in the infinite dimensional
case due to Lemma~\ref{lem:approximatedMeasure} and the fact that
\[
-\sup_{z\in B_{2R}(0)}\Phi(z)+\inf_{z\in B_{2R}(0)}\Phi(z)\leq -
\sup_{z\in B_{2R}(0)}\Phi(P_{n}z)+\inf_{z\in B_{2R}(0)}
\Phi(P_{n}z).
\]
Hence the claim follows.
\end{pf}

\subsection{Local log-Lipschitz density}
\label{sub:proofWassersteinLocal}

Now we allow the local Lipschitz constant
\[
\phi(r)=\sup_{x\neq y\in B_{r}(0)}\frac{\llvert \Phi
(x)-\Phi(y)\rrvert }{\llVert  x-y\rrVert }
\]
to grow in $r$. We used that $\Phi$ is globally Lipschitz to prove
that $\mathcal{P}$ and $\mathcal{P}_m$ is $d$-contracting; cf.
equation (\ref{GlobLipPhi}). Now there is no one fixed $\epsilon$
that makes $\mathcal{P}$ \mbox{$d$-}contracting. Instead the idea is to
change the metric in a way such that two points far out have
to be closer in $\llVert \cdot\rrVert _{\mathbf{\mathcal{H}}}$
in order to be considered ``close,'' that is, $d(x, y)<1$. This is inspired
by constructions in \citet{MR2602020,weakHarris}. Setting
\[
\mathsf{A}(T,x,y):=\bigl\{\psi\in C^{1}\bigl([0,T],\mathcal{H}\bigr),
\psi (0)=x,\psi(T)=y,\|\dot{\psi}\|=1\bigr\},
\]
we define the two metrics $d$ and $\bar{d}$ by
%
\begin{equation}\label{eq:localdistance}
\hspace*{20pt}d(x,y)=1\wedge\bar{d}(x,y),\qquad  \bar{d}(x,y)=\inf_{T,\psi\in
\mathsf{A}(T,x,y)}
\frac{1}{\epsilon}\int_{0}^{T}\exp\bigl(\eta
\llVert \psi\rrVert \bigr)\,dt,
\end{equation}
where $\epsilon$ and $\eta$ will be chosen depending on $\Phi$ and
$\gamma$ in the subsequent proof.
The situation is different from before because even in the case when
``both accept,'' the distance can increase because of the weight. In
order to control this, we notice the following:
%
\begin{lem}
\label{lem:metric}Let $\psi$ be a path connecting $x,y$ with $\|\dot
\psi\| = 1$,
then for $\bar{d}$ as in~(\ref{eq:localdistance}):
\begin{longlist}[(3)]
\item[(1)] $\frac{1}{\epsilon}\int_{0}^{T}\exp(\eta\llVert
\psi\rrVert )\,dt<1$
implies
\[
T\leq J:=\epsilon\exp \bigl(-\eta\bigl(\llVert x\rrVert \vee \llVert y\rrVert -
\epsilon\bigr)\vee0 \bigr)\leq\epsilon;
\]
\item[(2)] $\bar{d}(x,y)\leq\frac{\llVert  x-y\rrVert
}{\epsilon}\exp (\eta(\llVert  x\rrVert \vee\llVert  y\rrVert ) )$
and
\[
\frac{\llVert  x-y\rrVert }{\epsilon}\exp \bigl(\eta\bigl(\llVert x\rrVert \vee\llVert y\rrVert -J
\bigr)\vee0 \bigr)\leq\bar{d}(x,y)
\]
for all points such that $\bar{d}(x,y)<1$;

\item[(3)] for points such that $\bar{d}(x,y)<1$
\[
\frac{\bar{d}(p_{x},p_{y})}{\bar{d}(x,y)}\leq(1-2\delta)^{\sfrac
{1}{2}}e^{-\eta\rho [\llVert  x\rrVert \vee\llVert
y\rrVert +\eta(\llVert \sqrt{2\delta}\xi\rrVert
+J) ] }.\vadjust{\goodbreak}
\]
\end{longlist}
\end{lem}
\begin{pf}
In order to prove the first statement, we observe that
\[
\epsilon\geq\int_{0}^{T}e^{\eta\llvert \llVert  x\rrVert
\vee\llVert  y\rrVert -t\rrvert }\,dt\geq
Te^{\eta (\llVert  x\rrVert \vee\llVert  y\rrVert -T)\vee0}\geq Te^{\eta (\llVert  x\rrVert \vee\llVert
y\rrVert -\epsilon)\vee0 }.
\]
For the second part we denote by $\psi$ the line segment connecting $x$
and $y$ in order to obtain an upper bound $d(x,y)$. For the lower bound
we use $\llVert \psi\rrVert \geq(\llVert  x\rrVert
\vee\llVert  y\rrVert -J)\vee0$ from the first part combined
with the fact that $T\leq\epsilon$.
Using the second part we get
\begin{eqnarray*}
\bar{d}(p_{x},p_{y}) & \leq& \frac{1}{\epsilon}(1-2
\delta)^{\sfrac
{1}{2}}\llVert x-y\rrVert e^{\eta [(\llVert  x\rrVert \vee\llVert  y\rrVert )-\rho(\llVert  x\rrVert
\vee\llVert  y\rrVert )+\sqrt{2\delta}\llVert \xi\rrVert  ]}
\\
& \leq& (1-2\delta)^{\sfrac{1}{2}}e^{\eta [-\rho(\llVert
x\rrVert \vee\llVert  y\rrVert )+\sqrt{2\delta}\llVert \xi\rrVert +J ]}\frac{1}{\epsilon}\llVert x-y
\rrVert e^{\eta(\llVert  x\rrVert \vee\llVert
y\rrVert -J)}
\\
& \leq& (1-2\delta)^{\sfrac{1}{2}}e^{\eta [-\rho(\llVert
x\rrVert \vee\llVert  y\rrVert )+\sqrt{2\delta}\llVert \xi\rrVert +J ]}\bar{d}(x,y),
\end{eqnarray*}
which is precisely the required bound.
\end{pf}

\subsubsection{Lyapunov functions}

This condition neither depends on the distance function $d$ nor on
the Lipschitz properties of $\Phi$. Hence Lemma~\ref{lem:lyapunov}
applies.

\subsubsection{The $d$-contraction}
\label{sub:Contraction-in-the}

The main difference between local and global Lipschitz $\Phi$ is
proving that $\mathcal{P}$ and $\mathcal{P}_m$ is $d$-contracting.
%
\begin{lem}
\label{lem:localContraction}If $\Phi$ satisfies Assumptions \ref
{ass:acceptance}
and \ref{ass:LocalLipschitz}, then $\mathcal{P}$ and $\mathcal{P}_{m}$
are $d$-contracting for $d$ as in (\ref{eq:localdistance}) with
a contraction constant uniform in~$m$.
\end{lem}
\begin{pf}
First suppose $x,y\in B_{R}(0)$ with $d(x,y)<1$, and denote the event
$A= \{ \omega|\llVert \xi\rrVert \leq\frac{2R}{\sqrt {2\delta}} \} $.
First we choose $R$ large, before dealing with the case when $\eta$ is
small and when $\epsilon$ is small. We have
%
\begin{eqnarray}\label{eq:locContraction}
d\bigl(\mathcal{P}(x,\cdot),\mathcal{P}(y,\cdot)\bigr) & \leq& \mathbb {P}(A) [
\mathbb{P}(\mathrm{both\ accept}|A) (1-\tilde{\rho })d(x,y)\nonumber
\\
& & {}+ \bigl[\mathbb{P}(\mathrm{both\ reject}|A)d(x,y) \bigr]
\nonumber
\\
& & {}+\mathbb{E}\bigl(\bigl(\alpha(x,p_{x})\wedge\alpha
(y,p_{y})\bigr)d(p_{x},p_{y});A^{c}
\bigr)
\\
\nonumber& &{} +\mathbb{E}\bigl(\bigl(1-\alpha(x,p_{x})\vee\alpha
(y,p_{y})\bigr)d(x,y);A^{c}\bigr)
\\
& &{} +\mathbb{P}(\mbox{only one accepts})\cdot1,
\nonumber
\end{eqnarray}
where the first two lines deal with both accept and both reject in
the case of $A$, the third and fourth line consider the same case
in the event of $A^{c}$. The last line deals with the case when only
one accepts.
For the first two lines of equation (\ref{eq:locContraction}) we argue that
\[
\mathbb{P}(\mathrm{both\ accept}|A)\geq\inf_{x,z\in B_{3R}(0)}\mathbb {P}(
\mbox{accepts}\vert p_{x}=z)=\exp\bigl(-\Phi^{+}(3R)+
\Phi^{-}(3R)\bigr).
\]

If both are accepted, we know from Lemma~\ref{lem:metric} that
\begin{eqnarray*}
\frac{\bar{d}(p_{x},p_{y})}{\bar{d}(x,y)} & \leq& (1-2\delta )^{\sfrac{1}{2}}\exp \bigl(-\eta\rho \bigl(
\llVert x\rrVert \vee\llVert y\rrVert \bigr)+\eta\bigl(\llVert \sqrt {2\delta}\xi
\rrVert +J\bigr) \bigr)
\\
& \leq& (1-2\delta)^{\sfrac{1}{2}}e^{\eta(3R+J)}\leq(1-\tilde{\rho}),
\end{eqnarray*}
where the last step follows for $\eta$ small enough. Using the complementary
probability, we obtain the following estimate:
\[
\mathbb{P}(\mathrm{both\ reject}|A)\leq1-\mathbb{P}(\mathrm{both\ accept}|A).
\]
Combining both estimates, it follows that $\mathbb{P}(A)
(1-\mathbb{P}(\mathrm{both\ accept}|A)(1-\tilde{\rho}) )$
as coefficient in front of $d(x,y)$. In order to show that $\mathcal
{P}$ is $d$-contracting,
we have to prove that the expression in the third and fourth line of
equation (\ref{eq:locContraction}) is close to $\mathbb
{P}(A^{c})\cdot d(x,y)$. We notice that
\begin{eqnarray*}
&&\mathbb{E} \bigl(\bigl(1-\alpha(x,p_{x})\vee\alpha
(y,p_{y})\bigr)d(x,y);A^{c} \bigr)\\
&&\quad {}+\mathbb{E} \bigl(\bigl(
\alpha (x,p_{x})\wedge\alpha(y,p_{y})\bigr)d(p_{x},p_{y});A^{c}
\bigr)
\\
&&\qquad \leq\mathbb{E} \bigl(d(p_{x},p_{y})\vee
d(x,y);A^{c} \bigr)\leq \bar{d}(x,y)\mathbb{E}\frac{\bar{d}(p_{x},p_{y})}{\bar
{d}(x,y)}\vee1
\\
&&\qquad  \leq d(x,y)\int_{\sqrt{2\delta}\llVert \xi\rrVert
>2R}1\vee e^{\eta(\sqrt{2\delta}\llVert \xi\rrVert
+J)}\,d\gamma(\xi),
\end{eqnarray*}
where the last step followed by Lemma~\ref{lem:metric}. For small
$\eta$ the above is arbitrarily close to $\mathbb{P}(A^{c})\cdot d(x,y)$
by the dominated convergence theorem. By writing the integrand as
$\chi_{\sqrt{2\delta}\llVert \xi\rrVert >2R} (1\vee
\exp(\eta(\sqrt{2\delta}\llVert \xi\rrVert +J) ))$
and applying Lemma~\ref{lem:approximatedMeasure}, we conclude that
this estimate holds uniformly in $m$. Combining the first four lines,
the coefficient
in front of $d(x,y)$ is less than $1$ independently of $\epsilon$.
Only $\mathbb{P}(\mbox{only one accepts})\cdot1$ is left to bound
in terms of $d(x,y)$,
\begin{eqnarray*}
\mathbb{P}(\mbox{only one accepts}) & = & \int\bigl\llvert \alpha
(x,p_{x})-\alpha(y,p_{y})\bigr\rrvert \,d\gamma(\xi)
\\
& \leq& \int\bigl(\bigl\llvert \Phi(p_{x})-\Phi(p_{y})
\bigr\rrvert +\bigl\llvert \Phi (x)-\Phi(y)\bigr\rrvert \bigr)\,d\gamma(\xi)
\\
& \leq& \epsilon d(x,y)\int\bigl(\phi\bigl((1-\rho)R+\sqrt{2\delta}\llVert \xi
\rrVert \bigr)+\phi(R)\bigr)\,d\gamma(\xi).
\end{eqnarray*}
The integral above is bounded by Fernique's theorem. Hence for
$\epsilon$
small enough, we get an overall contraction when we combine this with
the result above.

Now let $x,y\in B_{\tilde{R}}^{c}(0)$ with $d(x,y)<1$, and without
loss of generality\vspace*{-1pt} we assume that $\llVert  y\rrVert \geq\llVert  x\rrVert $. Similar to the first case we bound
with $A=\{\omega|\|\sqrt{2\delta}\zeta\|\le r(\llVert  x\rrVert )\}$, we have
\begin{eqnarray*}
d\bigl(\mathcal{P}(x,\cdot),\mathcal{P}(y,\cdot)\bigr) & \leq& \mathbb {P}(A)
\bigl[\mathbb{P}(\mathrm{both\ accept}|A) (1-\rho)d(x,y)
\\
& &\hspace*{50pt}{} +\mathbb{P}(\mathrm{both\ reject}|A)d(x,y) \bigr]\\
& &{}+\mathbb {E}
\bigl(d(x,y)\vee d(p_{x},p_{y});A^{c} \bigr)
\\
& &{} +\mathbb{P}(\mbox{only one accepts})\cdot1.
\end{eqnarray*}
If ``both accept,'' then the contraction factor associated to the event
of $A$
is smaller than $(1-\rho)$ because $r(\llVert  x\rrVert )\leq
\frac{\rho}{2}\llVert  x\rrVert $
and by an application of Lemma~\ref{lem:metric}. For the next term it
follows that
\begin{eqnarray*}
\mathbb{E} \bigl(d(p_{x},p_{y})\vee d(x,y);A^{c}
\bigr) & \leq& \bar {d}(x,y)\mathbb{E}\frac{\bar{d}(p_{x},p_{y})}{\bar{d}(x,y)}\vee1
\\
& \leq& \bar{d}(x,y)\int_{A^{c}}1\vee e^{-\rho\eta
(\llVert  y\rrVert )+\eta(\llVert \sqrt{2\delta}\xi
\rrVert +J)}\,d\gamma(
\xi).
\end{eqnarray*}
Denoting the integral above by $I$, its integrand by $f(\xi)$
and $F>0$, this yields
\begin{eqnarray*}
I\leq I_{1}+I_{2}&=&\int_{\rho(\llVert  y\rrVert -J)+F\geq
\llVert \sqrt{2\delta}\xi\rrVert \geq r(\llVert  x\rrVert \wedge\llVert  y\rrVert )}f(\xi)\,d
\gamma(\xi )\\
&&{}+\int_{\llVert \sqrt{2\delta}\xi\rrVert \geq\rho
(\llVert  y\rrVert -J)+F}f(\xi)\,d\gamma(\xi).
\end{eqnarray*}
For the first part we have the upper bound $\mathbb{P}(A^{c})e^{\sqrt
{2\delta}\eta F}$ and for the second part we take $g\in X^{\star}$
with $\llVert  g\rrVert =1$.
We note that $\{x|g(x)>R\}\subseteq B_{R}(0)^{c}$ and hence
\[
\gamma\bigl(B_{R}(0)^{c}\bigr)\geq\gamma\bigl(\bigl
\{x|g(x)>R\bigr\}\bigr)\geq\exp\bigl(-\tilde {\beta}R^{2}+\zeta\bigr)
\]
using the one-dimensional lower bound. For the uniformity in $m$
we choose $g=e_{1}^{\star}$. We incorporate all occurring constants
into $\zeta$ and use Proposition~\ref{pro:splitMomentsEsitmates}
to bound
\begin{eqnarray*}
I_{2} & \leq& \mathbb{P}\bigl(A^{c}\bigr)\exp \biggl(
\tilde{\beta}\frac
{r(\llVert  x\rrVert )^{2}}{2\delta}-\rho\eta\bigl(\llVert y\rrVert -J\bigr)
\\
& &\hphantom{\mathbb{P}\bigl(A^{c}\bigr)\exp \biggl(} +\eta\sqrt{2\delta}\bigl(\rho\bigl(\llVert y\rrVert -J\bigr)+F\bigr)-
\beta\sqrt{2\delta}\bigl(\rho\bigl(\llVert y\rrVert -J\bigr)+F
\bigr)^{2}+\zeta \biggr).
\end{eqnarray*}
For any $\tau>0$ we choose $F$ large enough and then $\eta$
small enough so that $I\leq(1+\tau)\mathbb{P}(A^{c})$. Again
the estimates above are independent of $\epsilon$ which we choose
small in order to bound $\mathbb{P}(\mbox{only one accepts}|A^{c})$
in terms of $d(x,y)$. Calculating as above yields
\begin{eqnarray*}
& & \int\bigl\llvert \alpha(x,p_{x})-\alpha(y,p_{y})\bigr
\rrvert \,d\gamma(\xi)
\\
&&\qquad  \leq \int\bigl\llvert \Phi(x)-\Phi(y)\bigr\rrvert +\bigl\llvert
\Phi(p_{x})-\Phi(p_{y})\bigr\rrvert \,d\gamma(\xi)
\\
&&\qquad  \leq \int \bigl(\phi\bigl(\llVert y\rrVert \bigr)+\phi\bigl(\llVert
p_{x}\rrVert \vee\llVert p_{y}\rrVert \bigr)\bigr)\,d\gamma(\xi)
\llVert x-y\rrVert
\\
&&\qquad  \leq \biggl(M_{\kappa}e^{\kappa\llVert  y\rrVert
}+\int\phi\bigl((1-\rho)\llVert y
\rrVert +\sqrt {2\delta}\llVert \xi\rrVert \bigr)\,d\gamma(\xi) \biggr)\llVert x-y
\rrVert
\\
&&\qquad  \leq CM_{\kappa}\epsilon e^{-\eta(\llVert  x\rrVert \vee
\llVert  y\rrVert -\epsilon)\vee0+\kappa\llVert  y\rrVert }\bar{d}(x,y),
\end{eqnarray*}
where the last step follows using the upper bound for $\llVert
x-y\rrVert $
from Lemma~\ref{lem:metric}. Choosing $\kappa=\frac{\eta}{2}$ and
$\epsilon$ small enough, we can guarantee a uniform\vadjust{\goodbreak} contraction.
Checking line by line, the same is true for the $m$-dimensional approximation.
\end{pf}

\subsubsection{The $d$-smallness}

Analogously to the globally Lipschitz case, we have the following:
%
\begin{lem}
\label{lem:localSmallness}If $S$ is bounded, then $\exists n\in
\mathbb{N}$
and $0<s<1$ such that for all $x,y\in S$, $m\in\mathbb{N}$, and
for $d$ as in (\ref{eq:localdistance}),
\[
 d\bigl(\mathcal{P}_{m}^{n}(x,
\cdot),\mathcal{P}_{m}^{n}(y,\cdot)\bigr)\leq s \quad  \mbox{and}\quad
d\bigl(\mathcal{P}^{n}(x,\cdot),\mathcal{P}^{n}(y,\cdot )
\bigr)\leq s.
\]
\end{lem}
\begin{pf}
By Lemma~\ref{lem:globealLSmallness}, $d$ and $\|\cdot\|$ are comparable
on bounded sets. If $X_{0},Y_{0}\in B_{R}(0)$, and both algorithms
accept $n$ proposals in a row which are all elements of $B_{2R}(0)$,
then for $n$ large enough,
\[
d(X_{n},Y_{n})\leq\frac{\exp(\eta(2R+J))}{\epsilon}\diam (S) (1-2
\delta)^{n/2}\leq\frac{1}{2}.
\]
Hence the result follows analogue to Lemma~\ref{lem:globealLSmallness}.
\end{pf}

\section{Results concerning the sample-path average}\label{sample
path average}

In this section we focus on sample path properties of the pCN algorithm
which can be derived from the Wasserstein and the $L_{\mu}^{2}$-spectral
gap. We prove a strong law of large numbers, a CLT and a bound on the
MSE. This allows us to quantify the approximation of $\mu(f)$ by
\[
S_{n,n_{0}}(f)=\frac{1}{n}\sum_{i=1}^{n}f(X_{i+n_{0}}).
\]

\subsection{Consequences of the Wasserstein spectral gap}
\label{sub:conseqWasser}

The immediate consequences of a Wasserstein spectral gap are weaker
than the
results from the $L^{2}$-spectral gap because they apply to a smaller
class of observables, but they hold for the algorithm started at any
deterministic point.

\subsubsection{Change to a proper metric and implications for
Lipschitz functionals}
For the Wasserstein CLT [\citet{cltWasserstein}] we need a Wasserstein
spectral gap with respect
to a metric. The reason for this is that the Monge--Kantorovich duality
is used for its proof \citet{cltWasserstein}.
Recall that Theorem~\ref{thm:localVaryingBalls} yields a Wasserstein
spectral gap for the ``distance''
\begin{eqnarray*}
\tilde{d}(x,y) & =&\sqrt{\bigl(1+\llVert x\rrVert ^{i}+\llVert y\rrVert
^{i}\bigr) (1\wedge d)}\qquad  \mbox{where}
\\
d (x,y)& =&\inf_{T,\psi\in\mathsf{A}(T,x,y)}\frac{1}{\epsilon
}\int_{0}^{T}
\exp\bigl(\eta\llVert \psi\rrVert \bigr).
\end{eqnarray*}
Because $\tilde{d}$ does not necessarily satisfy the triangle inequality,
we introduce
%
\begin{eqnarray}\label{eq:metric}
d^{\prime}(x,y) & =&\sqrt{\mathop{\inf_{
x=z_{1},\dots, z_{n}=y}}_{n\ge2}
\sum
_{j=1}^{n-1}d_{0}(z_{j},z_{j+1})},
\nonumber\\
d_{0}(x,y) & =&d_{1}(x,y)\wedge d_{2}(x,y),
\nonumber
\\
d_{1}(x,y) & =& %
\cases{ 0, &\quad  $x=y$,
\cr
\bigl(1+\llVert x
\rrVert ^{i}+\llVert y\rrVert ^{i}\bigr), & \quad \mbox{otherwise},
} %
\\
d_{2}(x,y) & =&\inf_{T,\psi\in\mathsf{A}(T,x,y)}F(\psi),
\nonumber
\\
F(\psi)&=&\frac{1}{\epsilon}\int_{0}^{T}\exp\bigl(
\eta\llVert \psi \rrVert \bigr) \bigl(1+\llVert \psi\rrVert ^{i}
\bigr)\,dt.
\nonumber
\end{eqnarray}

It is straightforward to verify that $d^{\prime}$ is a metric by
first showing that the expression inside the square root is a metric
(triangle inequality is satisfied because of the infimum) and using
that a square root of a metric is again a metric.

Moreover, $\mathcal{P}$ and $\mathcal{P}^{m}$ have a Wasserstein
spectral gap with respect to $d^{\prime}$ because of the following lemma:
%
\begin{lem} \label{lem:LocalLogLipMetric}
Provided that $\epsilon$ is small enough, there exists a constant
$C>0$ such that
\[
d^{\prime}(x,y)\leq\tilde{d}(x,y)\leq Cd^{\prime}(x,y)
\]
for all pairs of points $x,y$ in $\mathcal{H}$.
\end{lem}
\begin{pf}
Without loss of generality we assume that $\llVert  y\rrVert
\geq\llVert  x\rrVert $.
The inequality $d'\le\tilde{d}$ follows from Lemma~\ref{lem:auxDistance}
since $d^{\prime}\leq\sqrt{d_{0}}$ by definition.

In order show that $\tilde{d}\leq Cd^{\prime}$, we will use
Lemma~\ref{lem:auxDistance}
and reduce the number of summands appearing in equation \eqref{eq:metric}
for $d^{\prime}$.
We can certainly assume that there is at most one index $j$ in \eqref
{eq:metric}
such that $d_{0}(z_{j},z_{j+1})=d_{1}(z_{j},z_{j+1})$ because otherwise
there are $1\leq j<k\leq n$ such that
\[
d_{0}(z_{j},z_{j+1})=d_{1}(z_{j},z_{j+1}),\qquad
d_{0}(z_{k},z_{k+1})=d_{1}(z_{k},z_{k+1})
\]
which would lead to
\[
d_{0}(z_{j},z_{j+1})+\cdots+d_{0}(z_{k},z_{k+1})
\ge2+\llVert z_{j}\rrVert ^{i}+\llVert z_{k+1}
\rrVert ^{i}>d_{1}(z_{j},z_{k+1}).
\]
Hence the expression could be made smaller by removing all intermediate points
between $z_j$ and $z_{k+1}$, a contradiction.

Because $d_{2}$ is a Riemannian metric, it satisfies the triangle
inequality in a sharp way in the sense that $d_2(x,y) = \inf_z
(d_2(x,z) + d_2(z,y) )$.
As a consequence, the infimum is not changed by assuming that in
equation \eqref{eq:metric}
there is no index $j$ such that
\[
d_{0}(z_{j},z_{j+1})=d_{2}(z_{j},z_{j+1}),\qquad
d_{0}(z_{j+1},z_{j+2})=d_{2}(z_{j+1},z_{j+2}).
\]

Combining these two facts, equation \eqref{eq:metric} thus reduces to
%
\begin{eqnarray}\label{eq:metricInfExpression}
\bigl(d'(x,y)\bigr)^2& = & \min \Bigl\{
d_{0}(x,y),\inf_{z_{2},z_{3}}d_{2}(x,z_{2})+d_{1}(z_{2},z_{3})+d_{2}(z_{3},y),
\nonumber
\\[-8pt]\\[-8pt]
&&\hphantom{\min \Bigl\{} \inf_{z_{2}}d_{2}(x,z_{2})+d_{1}(z_{2},y),
\inf_{z_{2}}d_{1}(x,z_{2})+d_{2}(z_{2},y)
\Bigr\}.\nonumber
\end{eqnarray}
Recalling Lemma~\ref{lem:auxDistance},
it remains to show that $d' \ge C \sqrt d_0$ with $d'$
given by \eqref{eq:metricInfExpression}. This is of course nontrivial
only if
$(x,y)$ is such that $d'(x,y) < \sqrt{d_0(x,y)}$. Therefore we assume
this fact from now on.

Suppose first that $\llVert  y\rrVert \leq Q$,
for some constant $Q>0$ to be determined later.
Since $d'(x,y) \neq\sqrt{d_0(x,y)}$, there is at
least one $j$ such that $d_{0}(z_{j},z_{j+1})=d_{1}(z_{j},z_{j+1})$
which leads to
\[
1+2Q^{i}\ge d_{0}(x,y)\ge \bigl(d^{\prime}(x,y)
\bigr)^2\ge1,
\]
so that the bound $(1+2Q^{i})d^{\prime}(x,y)\ge\sqrt{d_{0}(x,y)}$
indeed follows
in this case.

Suppose now that $\llVert  y\rrVert \ge Q$.
Again, one
summand $d_{0}(z_{j},z_{j+1})$ in equation \eqref{eq:metricInfExpression}
satisfies
\[
d_{0}(z_{j},z_{j+1})=d_{1}(z_{j},z_{j+1}),
\]
thus giving rise to a simple lower bound on $d^{\prime}$,
%
\begin{equation}
\label{e:lowerBound} d^{\prime}(x,y)\ge\sqrt{1+\llVert z_j\rrVert
^i}.
\end{equation}

Because of \eqref{eq:metricInfExpression}, $z_{j+1}$ is either equal
to $y$ or connected to
$y$ through a path $\psi_y \in\mathsf{A}(T,z_{j+1},y)$ which is such that
%
\begin{equation}
\label{boundF} F(\psi_y)\leq1+2 \llVert y\rrVert ^i,
\end{equation}
where $F(\psi)$ is as in the definition of $d_2$.
By the same reasoning as in the proof of Lemma~\ref{lem:auxDistance},
for $Q$ large enough
it is sufficient to consider paths starting in $y$ and such that
$\llVert \psi(t)\rrVert \ge\llVert  y\rrVert / 2$.
The bound \eqref{boundF} thus yields an upper bound on $\|z_{j+1} - y\|
$ by
%
\begin{equation}\label{eq:wassersteinMetricLemmaLengthOfPath}
1+2\llVert y\rrVert ^{i} \geq F(\psi_y) \ge
\frac
{1}{\epsilon}\|z_{j+1} - y\| \exp\bigl(\eta\llVert y\rrVert /2
\bigr). 
\end{equation}
Combining this with \eqref{e:lowerBound}, we have
\begin{eqnarray*}
d'(x,y)^2 & \ge&1+\bigl(\llVert y\rrVert -
\|z_{j+1} - y\|\bigr)^{i}\ge 1+\llVert y\rrVert
^{i}-i\llVert y\rrVert ^{i-1}\| z_{j+1} - y\|
\\
& \ge&1+\frac{\llVert  y\rrVert ^{i}}{2}+ \biggl(\frac{\llVert  y\rrVert ^{i}}{2}-\epsilon\bigl(1+2\llVert y
\rrVert ^{i}\bigr)\exp\bigl(-\eta\llVert y\rrVert /2\bigr) \biggr),
\end{eqnarray*}
provided that $\epsilon< 1/4$ and $Q$ is large enough,
the third summand is positive so that
$d'(x,y)^2\ge{1\over4} d_1(x,y) \ge{1\over4} d_0(x,y)$ completing
the proof.
\end{pf}
%
\begin{lem}\label{lem:auxDistance}
There is a $C>0$ such that $d_{0}$ as
defined in equation \eqref{eq:metric} satisfies
\[
d_{0}(x,y)\le\tilde{d}(x,y)^{2}\leq Cd_{0}(x,y)
\qquad \mbox{for all }x,y.
\]
\end{lem}
\begin{pf}
We assume again that $\llVert  y\rrVert \geq\llVert
x\rrVert $.
In order to prove that $d_{0}(x,y)\le\tilde{d}(x,y)^{2}$, we only
have to show that
\[
\inf_{T,\psi\in\mathsf{A}(T,x,y)}F(\psi)\le\inf_{T,\psi\in\mathsf{A}(T,x,y)}
\frac{1}{\epsilon}\int_{0}^{T}\exp\bigl(\eta\llVert
\psi\rrVert \bigr)\,dt\bigl(1+\llVert x\rrVert ^{i}+\llVert y\rrVert
^{i}\bigr).
\]
Replacing $\psi(t)$ by
\[
\bigl(1\wedge\|y\|/\bigl\|\psi(t)\bigr\|\bigr) \psi(t)
\]
in the expressions above does not cause an increase. Hence
it is sufficient to consider paths $\psi$ which satisfy
%
\begin{equation}\label{eq:provMetricReflect}
\bigl\llVert \psi(t)\bigr\rrVert \leq\llVert y\rrVert, \qquad t\in
[0,T].
\end{equation}
The bound $d_{0}\leq\tilde{d}^{2}$ then follows at once from
\[
1+\llVert \psi\rrVert ^{i}\leq1+\llVert x\rrVert ^{i}+
\llVert y\rrVert ^{i}.
\]

We proceed now to show that $\tilde{d}(x,y)^{2}\leq Cd_{0}(x,y)$ for which
we only have to consider
%
\begin{eqnarray}\label{eq:functionalUpperBound}
d_{2}(x,y)&=&\inf_{T,\psi\in\mathsf{A}(T,x,y)}\frac
{1}{\epsilon}\int
_{0}^{T}\exp\bigl(\eta\llVert \psi\rrVert \bigr)
\bigl(1+\llVert \psi\rrVert ^{i}\bigr)\,dt\nonumber\\[-8pt]\\[-8pt]
&\leq&\bigl(1+\llVert x\rrVert
^{i}+\llVert y\rrVert ^{i}\bigr)\nonumber
\end{eqnarray}
since the minimum expressions in $\tilde{d}^{2}$ and $d_{0}$ have
$(1+\llVert  x\rrVert ^{i}+\llVert  y\rrVert ^{i})$ in
common.

We will first use this to show that $x$ and $y$ have to be
close if $\llVert  y\rrVert $ is large. We will show that any
path $\psi$ for which the expression in $d_{2}$ is close to the
infimum has to satisfy $\llVert  y\rrVert \geq\psi
\geq\frac{\llVert  \mathsf{y}\rrVert }{2}$.
Hence $1+\llVert \psi\rrVert ^{i}$ and $(1+\llVert
x\rrVert ^{i}+\llVert  y\rrVert ^{i})$
are comparable. In order to gain a lower bound on $d_{2}(x,y)$, we
distinguish between paths $\psi$ which intersect or do not intersect
$B_{R}(0)$. If the path lies completely outside the ball, we have
\[
d_{2}(x,y)\geq\frac{1}{\epsilon}\llVert x-y\rrVert \exp (\eta R)
\bigl(1+R^{i}\bigr).
\]

If $\psi$ and $B_{R}(0)$ intersect, then $d_{2}(x,y)$
is larger than $d_{2}(B_{R}(0),y)$ which by symmetry corresponds
to
\begin{eqnarray*}
d_{2}(x,y) & \geq& \frac{1}{\epsilon}\int_{0}^{\llVert  y\rrVert -R}
\exp\bigl(\eta\bigl(\llVert y\rrVert -t\bigr)\bigr) \bigl(1+\bigl(\llVert y\rrVert
-t\bigr)^{i}\bigr)\,dt
\\
& \geq& \frac{1}{\epsilon}\bigl(\llVert y\rrVert -R\bigr)\exp\bigl(\eta \bigl(\llVert
y\rrVert -R\bigr)\bigl(1+\bigl(\llVert y\rrVert -R\bigr)^{i}\bigr)\bigr).
\end{eqnarray*}
We choose $R=\frac{\llVert  y\rrVert }{2}$ and note
that $\frac{\llVert  y\rrVert }{2}\geq\frac{\llVert
x-y\rrVert }{4}$, leading in both cases to
\[
d_{2}(x,y)\ge\frac{1}{4\epsilon}\llVert x-y\rrVert \exp \bigl(\eta
\llVert y\rrVert /2\bigr) \biggl(1+\frac{\llVert  y\rrVert
^{i}}{2}\biggr).
\]
By (\ref{eq:functionalUpperBound}) this implies
%
\begin{equation}\label
{eq:MetricForWasserstein}
\llVert x-y\rrVert \leq\frac{4\epsilon\exp(-\eta\sfrac
{\llVert  y\rrVert }{2})}{1+\llVert  y\rrVert
^{i}/2} \bigl(1+2\llVert y\rrVert
^{i} \bigr).
\end{equation}
For $x$ and $y$ in $B_{Q}(0)$ we have that $(\tilde{d})^{2}\leq
(2Q^{i}+1)d_{0}$
because we can assume $\llVert \psi(t)\rrVert \leq\llVert  y\rrVert $ as above.
It is only left to consider $x,y\in B_{\tilde{Q}}(0)^{c}$ for\vspace*{-1pt} $\tilde
{Q}=Q-4\epsilon\exp(-\eta\frac{Q}{2})(1+2Q^{i})$
because of equation (\ref{eq:MetricForWasserstein}). Subsequently, we
will show that for $Q$ and hence $\tilde{Q}$
large enough, it is sufficient for the infimum expression for $d_{2}$
to consider paths $\psi$ that do not intersect $B_{R}(0)$ for $R=\frac
{\llVert  y\rrVert }{2}$.

Suppose that the path $\psi$ would intersect $B_{R}(0)$. Then the functional
is larger than the shortest path $\hat{\psi}$ to the boundary of
the ball and hence
%
\begin{eqnarray}\label{eq:getMetricLower}
d_{2} & \geq& F(\hat{\psi})\ge\frac{1}{\epsilon}\int
_{0}^{\llVert  y\rrVert -R}e^{\eta(\llVert  y\rrVert -t)}\bigl(1+\bigl(\llVert y
\rrVert -t\bigr)^{i}\bigr)\,dt
\nonumber
\\
& = & \frac{1}{\epsilon} \Biggl[\exp\bigl(\eta\llVert y\rrVert \bigr) \Biggl(
\eta^{-1}\bigl(1+\llVert y\rrVert ^{i}\bigr)+\sum
_{j=1}^{i}\eta ^{-1-j}\frac{i!}{(i-j)!}
\llVert y\rrVert ^{i-j}\Biggr)
\\
& &\hspace*{26pt}{} -\exp(\eta R) \Biggl(\eta^{-1}\bigl(1+R^{i}
\bigr)+\sum_{j=1}^{i}\eta ^{-1-j}
\frac{i!}{(i-j)!}R^{i-j}\Biggr) \Biggr]\nonumber
\end{eqnarray}
by $i+1$ integrations by parts. Let $l$ be the line connecting
$x$ and $y$. Then using (\ref{eq:MetricForWasserstein}) yields
\[
F(l)\leq\frac{1}{\epsilon}\llVert x-y\rrVert e^{\eta\llVert  y\rrVert }\bigl(1+\llVert y
\rrVert ^{i}\bigr)\leq4\exp\biggl(\eta \frac{\llVert  y\rrVert }{2}\biggr)
\bigl(1+2\llVert y\rrVert ^{i}\bigr)^{2}.
\]
For $R=\frac{\llVert  y\rrVert }{2}$ and $\tilde{Q}$ large
enough we have $F(\psi)>F(l)$. Therefore for all $t\in[0,T]$ $\llVert  y\rrVert \geq\psi\geq\llVert  \mathsf{y}\rrVert /2$
and thus
\[
2^{i+1}\bigl(1+\llVert \psi\rrVert ^{i}\bigr)\geq\bigl(1+
\llVert x\rrVert ^{i}+\llVert y\rrVert ^{i}\bigr)
\]
which yields that $\max(2Q^{i}+1,2^{i+1})d_{0}\geq\tilde{d}^{2}$.
\end{pf}

\subsubsection{Strong law of large numbers}

In this section we will prove a strong law of large numbers for Lipschitz
functions. Since $\mu_m$ ($\mu)$ are the unique invariant measures
for $\mathcal{P}$ ($\mathcal{P}_{m})$ (resp.), $\mu_m$ ($\mu$) is ergodic
and Birkhoff's ergodic theorem applies. However, this theorem only
applies to almost every initial condition, but we are able to extend it
to every initial condition in this case which yields a strong law of
large numbers.
%
\begin{theorem}
\label{thm:slln}In the setting of Theorem~\ref{thm:global} or \ref
{thm:localVaryingBalls},
suppose that $\operatorname{supp}\mu=\mathcal{H}$ and $h\dvtx \mathbf
{\mathcal{H}}\rightarrow\mathbb{R}$
has Lipschitz constant $L$ with respect to $\tilde{d}$, then for
arbitrary $X_{0}\in\mathcal{H}$
\[
\Biggl\llvert \frac{1}{n}\sum_{i=1}^{n}h
\bigl(X^{i}\bigr)-\mathbb{E}_{\mu}h\Biggr\rrvert \stackrel{
\mathrm{a.s.}} {\rightarrow}0.
\]
\end{theorem}
\begin{pf}
By Birkhoff's ergodic theorem, we know that this is true for measurable
$h$ and every initial condition in some set of full measure $A$.
Because $\mu$ has full support, for
any $t>0$ we can choose $Y_{0} \in A$ with $\tilde{d}(X_{0},Y_{0})\leq t^{2}$.
Hence
\begin{eqnarray*}
\Biggl\llvert \frac{1}{n}\sum_{i=1}^{n}h
\bigl(X^{i}\bigr)-\mathbb{E}_{\mu}h\Biggr\rrvert & \leq&
\Biggl\llvert \frac{1}{n}\sum_{i=1}^{n}h
\bigl(Y^{i}\bigr)-\mathbb{E}_{\mu
}h\Biggr\rrvert +\Biggl
\llvert \frac{1}{n}\sum_{i=1}^{n}
\bigl(h\bigl(X^{i}\bigr)-h\bigl(Y^{i}\bigr)\bigr)\Biggr
\rrvert
\\
& \leq& \Biggl\llvert \frac{1}{n}\sum_{i=1}^{n}h
\bigl(Y^{i}\bigr)-\mathbb{E}_{\mu
}h\Biggr\rrvert +
\frac{1}{n}\sum_{i=1}^{n}L\tilde{d}
\bigl(X^{i},Y^{i}\bigr).
\end{eqnarray*}

By the Wasserstein spectral gap, we can couple $X^{n}$ and $Y^{n}$
such that
\[
\mathbb{E}\tilde{d}\bigl(X^{n},Y^{n}\bigr)\leq
Cr^{n}\tilde{d}\bigl(X^{0},Y^{0}\bigr)
\]
for some $0<r<1$. An application of Markov's inequality yields
\[
\mathbb{P} \bigl(\tilde{d}\bigl(X^{n},Y^{n}\bigr)\geq c
\bigr)\leq C\frac
{r^{n}\tilde d(X^{0},Y^{0})}{c}.
\]
Since Birkhoff's theorem applies to the Markov process started at
$Y_{0}$, we have
\begin{eqnarray*}
\mathbb{P} \Biggl[\lim\sup\Biggl\llvert \frac{1}{n}\sum
_{i=1}^{n}h\bigl(X^{i}-
\mathbb{E}_{\mu}h\bigr)\Biggr\rrvert \geq c \Biggr]& =&\mathbb {P} \Biggl[
\lim\sup\frac{1}{n}\sum_{i=1}^{n}\bigl
\llvert h\bigl(X^{i}\bigr)-h\bigl(Y^{i}\bigr)\bigr\rrvert
\geq c \Biggr]
\\
& \leq& C\frac
{L}{c(1-r)}\tilde d\bigl(X^{0},Y^{0}\bigr).
\end{eqnarray*}
Setting $c=\frac{t}{L}$ yields
\[
\mathbb{P} \Biggl(\lim\sup\Biggl\llvert \frac{1}{n}\sum
_{i=1}^{n}h\bigl(X^{i}-
\mathbb{E}_{\mu}h\bigr)\Biggr\rrvert \leq t \Biggr)\geq 1-t
\frac{C}{1-r},
\]
and because $t$ was chosen arbitrarily, the result follows.
\end{pf}


\subsubsection{Central limit theorem}

The result above does not give any rate of convergence. With a CLT
on the other hand it is possible to derive (asymptotic) confidence
intervals and to estimate the error for finite $n$. Because of
Lemma~\ref{lem:LocalLogLipMetric} and arguments from Lemma~\ref{lem:lyapunov},
it is straightforward to verify that our assumptions imply those needed
for the Wasserstein CLT in \citet{cltWasserstein}. This results in the
following theorem:
%
\begin{theorem}
\label{thm:wassersteinclt}If the conditions of Theorem~\ref{thm:global}
or \ref{thm:localVaryingBalls} are satisfied, then there exists
$\sigma\in[0,+\infty)$
such that
\[
\lim_{n\rightarrow+\infty}\frac{1}{n}\mathbb{E} \Biggl(\sum
_{i=1}^{n}\tilde{f}(X_{s})
\Biggr)^{2}=\sigma^{2},
\]
where $\tilde{f}:=f-\mu(f)$ and f is Lipschitz with respect to
$d^{\prime}$.
Moreover, we have
\[
\lim_{T\rightarrow\infty}\mathbb{P}\Biggl(\frac{1}{\sqrt{n}}\sum
_{i=1}^{n}\tilde{f}(X_{s})<\xi\Biggr)=
\Phi_{\sigma}(\xi)\qquad  \forall\xi \in\mathbb{R},
\]
where $\Phi_{\sigma}(\cdot)$ is the distribution function of
$\mathcal{N}(0,\sigma^{2})$
a zero mean normal law whose variance equals $\sigma^{2}$.
\end{theorem}

\subsection{Consequences of \texorpdfstring{$L_{\mu}^{2}$}{L mu 2}-spectral gap}
\label{sub:conseqL2}

Under the assumptions of\vspace*{-1pt} Theorem \ref{thm:global} or \ref
{thm:localVaryingBalls},
we have proved the existence of an $L_{\mu}^{2}$-spectral gap in Section~\ref{sub:L2}. Now we may use all existing consequences for the ergodic
average with and without burn in ($n_{0}=0$)
\[
S_{n,n_{0}}(f)=\frac{1}{n}\sum_{j=1}^{n}f(X_{j+n_{0}}),\qquad
S_{n}=S_{n,0}.
\]
The following result of \citet{kipnis1986central} [see also \citet
{Latuszynskiclt}
whence the statement was adapted] then yields a CLT:
%
\begin{prop}\label{kipnisvara-clt}
Consider an ergodic Markov chain with transition operator $P$ which is
reversible with respect to a probability measure $\mu$. Let
$f\in L^{2}$ be such that
\[
\sigma_{f,P}^{2}= \biggl\langle\frac{1+P}{1-P} f,f \biggr
\rangle <\infty,
\]
and then for $X_{0}\sim\mu$ the expression $\sqrt{n}(S_{n}-\mu(f))$
converges weakly to $\mathcal{N}(0,\sigma_{f,P}^{2})$.
\end{prop}
In our case, provided that $f$ is mean-zero,
it follows from the $L^2$-spectral gap that
\[
\sigma_{f,\mathcal{P}}^{2} \le\frac{2\mu(f^{2})}{1-\beta}.
\]

Due to Theorem~\ref{thm:localVaryingBalls}, we have a lower bound on
the spectral gap $1-\beta$ of $\mathcal{P}$ and $1-\beta_m$ of
$\mathcal{P}_m$ which is uniform in $m$.
Thus the ergodic average of the pCN algorithm applied to the target
measures $\mu$ and $\mu_m$
has an $m$-independent upper bound on the asymptotic variance.

The result of Proposition~\ref{kipnisvara-clt} has been extended to
$\mu$ for almost every initial
condition in \citet{cuny2009pointwise} which also applies to our
case.\\

A different approach due to \citet{explicitbdd} is to consider the
MSE
\[
e_{\nu}(S_{n,n_{0}},f)=\bigl(\mathbb{E}_{\nu,K}\bigl
\llvert S_{n,n_{0}}(f)-\mu (f)^{2}\bigr\rrvert
\bigr)^{1/2}.
\]
Using Chebyscheff's inequality, this results in a confidence interval
for $S(f)$. We can bound it by using the following proposition from
\citet{explicitbdd}:
%
\begin{prop}
Suppose that we have a Markov chain with Markov operator $\mathcal{P}$
which has an $L_{\mu}^{2}$-spectral gap $1-\beta$. For $p\in
(2,\infty]$
let $n_{0}(p)$ be the smallest natural number which is greater or
equal to
%
\begin{equation}\label{eq:burnin}
\qquad\frac{1}{\log(\beta^{-1})} %
\cases{ \displaystyle \frac{p}{2(p-2)}\log\biggl(
\displaystyle \frac{32p}{p-2}\biggr)\biggl\llVert \displaystyle \frac{d\nu}{d\mu
}-1\biggr\rrVert
_{\afrac{p}{p-2}}, &\quad  $p\in(2,4)$,
\cr
\log(64)\biggl\llVert \displaystyle \frac{d\nu}{d\mu}-1
\biggr\rrVert _{\afrac
{p}{p-2}}, &\quad  $p\in[4,\infty]$. } %
\end{equation}
Then
\[
\sup_{\llVert  f\rrVert _{p}\leq1}e_{\nu
}(S_{n,n_{0}},f)\leq
\frac{2}{n(1-\beta)}+\frac{2}{n^{2}(1-\beta)^{2}}.
\]
\end{prop}
In our setting $n_{0}(p)$ is finite for $\nu=\gamma$ under the additional
assumption that for all $u_{1}>0$ there is a $u_{2}$ such that
\[
\Phi\bigl(\llVert x\rrVert \bigr)\leq u_{1}\llVert x\rrVert
^{2}+u_{2}.
\]
Using Fernique's theorem, this implies that $\frac{d\gamma}{d\mu}-1$
has moments of all orders.

\section{Conclusion}
\label{sec:Conclusion}

From an applications perspective, the primary thrust of this paper
is to develop an understanding of MCMC methods in high dimension.
Our work has concentrated on identifying the (possibly lack of) dimension
dependence of spectral gaps for the standard random walk method RWM
and a recently developed variant pCN adapted to measures defined via a
density with respect to a Gaussian.
It is also possible to show that the function space version of the MALA
\citet{beskos2012advanced} has a spectral gap if, in addition to the
assumptions in this article, the gradient of $\Phi$ satisfies strong
assumptions, and the gradient step is very small. There is also a
variant of the hybrid Monte Carlo methods \citet{Con1} adapted to the
sampling of measures
defined via a density with respect to a Gaussian and it would be interesting
to employ the weak Harris theorem to study this algorithm.

Other classes of target measures, such as those arising from Besov prior
measures [\citet{Con2,Con3}] or an infinite product of uniform
measures in \citet{Con4}
would also provide interesting applications.
The proposal of the pCN is reversible and has a spectral gap with
respect to the Gaussian reference measure. For arbitrary reference and
target measures, the third author has recently proved that for bounded
$\Phi$ the Metropolis--Hastings algorithm has a spectral gap if the
proposal is reversible and has a spectral gap with respect to the
reference measure [\citet{MCMCnote}].

More generally, we expect that the weak Harris theorem will be
well suited to the study of many
MCMC methods in high dimensions because of its roots in the study
of Markov processes in infinite dimensional spaces [\citet{weakHarris}].
In contrast, the theory developed in \citet{Meyn:2009uqa} does not
work well for the kind of high dimensional problems that are studied
here.

From a methodological perspective, we have demonstrated a particular
application of the theory developed in \citet{weakHarris}, demonstrating
its versatility for the analysis of rates of convergence in Markov
chains. We have also shown how that theory, whose cornerstone is a
Wasserstein spectral gap, may usefully be extended to study
$L^{2}$-spectral gaps and resulting sample path properties. These
observations will be useful in a variety of applications, not just
those arising
in the study of MCMC.\looseness=1

All our results were presented for separable Hilbert spaces, but in
fact they do also hold on an arbitrary Banach space. This can be shown
by using a
Gaussian series [cf. Section~3.5 in \citet{gaussianMeasureas}] instead
of the Karhunen--Lo\`{e}ve expansion and the $m$-independence is due to
Theorem~3.3.6 in \citet{gaussianMeasureas}.

\begin{appendix}
\section*{Appendix: Gaussian measures}
As a consequence of Fernique's theorem,
we have the following explicit bound on exponential moments of the norm
of a Gaussian random
variable, which is needed to show that $\mathcal{P}$ and $\mathcal
{P}_{m}$ are
$d$-contracting; see Section~\ref{sub:Contraction-in-the}.

\begin{prop}
\label{pro:splitMomentsEsitmates}
For $\beta$ small enough, there exists a constant $F_\beta$ such that
\[
\int_{X}\exp\bigl(\beta\llVert u\rrVert ^{2}
\bigr)\,d\gamma (u)=F_{\beta}.
\]
Furthermore, for any $\alpha\in\mathbb{R}^{+}$ there
is a constant $C_{\alpha,\beta}$ such that for $K>\frac{\alpha
}{2\beta}$
\[
\int_{\{\llVert  u\rrVert \geq K\}}\exp\bigl(\alpha\llVert u\rrVert \bigr)\,d\gamma(u)
\leq C_{\alpha,\beta}e^{-\beta K^{2}+\alpha K}.
\]
\end{prop}
\begin{pf}
The first claim is just Fernique's theorem; see, for example, \citet
{gaussianMeasureas,StochEqnInf,hairer}.
Using integration by parts and Fubini, we get
\[
\int_{\llVert  x\rrVert \geq K}f\bigl(\llVert x\rrVert \bigr)\,d\gamma = f(K)\gamma
\bigl(\llVert x\rrVert \geq K\bigr)+\int_{K}^{\infty}
\gamma\bigl(\llVert x\rrVert \geq t\bigr)f^{\prime}(t)\, dt.
\]
Setting $f(x)=\exp(\alpha x)$ and applying Fernique's theorem yields
\[
\int_{\llVert  x\rrVert \geq K}\exp\bigl(\alpha\llVert x\rrVert \bigr)\,d\gamma
\leq F_{\beta}\exp\bigl(-\beta K^{2}+\alpha K
\bigr)+F_{\beta}\alpha\int_{K}^{\infty}\exp
\bigl(-\beta t^{2}+\alpha t\bigr)\, dt.
\]
Since, for $K$ as in the statement, one verifies that
\[
\beta t^{2}-\alpha t \ge\beta K^2 - \alpha K +
\beta(t-K)^2,
\]
and the required bound follows at once.
%
\end{pf}
\end{appendix}

\section*{Acknowledgments}
We thank Feng-Yu Wang for pointing out the connection between
Wasserstein and $L^{2}$-spectral gaps and Professor Andreas Eberle for
many fruitful discussions about this topic.



\printaddresses

\end{document}